\documentclass[leqno,11pt]{amsart}

\usepackage{amssymb, amsmath}
\usepackage{amssymb,amsfonts}
\usepackage{amsmath,latexsym}
\usepackage{pdfsync}
\usepackage{bbm}

\newtheorem{theorem}{Theorem}[section]
\newtheorem{proposition}{Proposition}[section]
\newtheorem{lemma}{Lemma}[section]

\newtheorem{corollary}{Corollary}[section]

\newcommand\mydef{\mathrel{\stackrel{\makebox[0pt]{\mbox{\small def}}}{=}}}

\begin{document}


\makeatletter
\def\sommaire{\@restonecolfalse\if@twocolumn\@restonecoltrue\onecolumn
\fi\chapter*{Sommaire\@mkboth{SOMMAIRE}{SOMMAIRE}}
  \@starttoc{toc}\if@restonecol\twocolumn\fi}
\makeatother

\makeatletter
\def\thebibliographie#1{\chapter*{Bibliographie\@mkboth
  {BIBLIOGRAPHIE}{BIBLIOGRAPHIE}}\list
  {[\arabic{enumi}]}{\settowidth\labelwidth{[#1]}\leftmargin\labelwidth
  \advance\leftmargin\labelsep
  \usecounter{enumi}}
  \def\newblock{\hskip .11em plus .33em minus .07em}
  \sloppy\clubpenalty4000\widowpenalty4000
  \sfcode`\.=1000\relax}
\let\endthebibliography=\endlist
\makeatother

\makeatletter
\def\references#1{\section*{R\'ef\'erences\@mkboth
  {R\'EF\'ERENCES}{R\'EF\'ERENCES}}\list
  {[\arabic{enumi}]}{\settowidth\labelwidth{[#1]}\leftmargin\labelwidth
  \advance\leftmargin\labelsep
  \usecounter{enumi}}
  \def\newblock{\hskip .11em plus .33em minus .07em}
  \sloppy\clubpenalty4000\widowpenalty4000
  \sfcode`\.=1000\relax}

\let\endthebibliography=\endlist
\makeatother

\newcommand*\wbar[1]{%
   \hbox{%
     \vbox{%
       \hrule height 0.5pt 
       \kern0.5ex
       \hbox{%
         \kern-0.25em
         \ensuremath{#1}%
         \kern-0.1em
       }%
     }%
   }%
} 

\def\lunloc#1#2{L^1_{loc}(#1 ; #2)}
\def\bornva#1#2{L^\infty(#1 ; #2)}
\def\bornlocva#1#2{L_{loc}^\infty(#1 ;\penalty-100{#2})}
\def\integ#1#2#3#4{\int_{#1}^{#2}#3d#4}
\def\reel#1{\R^#1}
\def\norm#1#2{\|#1\|_{#2}}
\def\normsup#1{\|#1\|_{L^\infty}}
\def\normld#1{\|#1\|_{L^2}}
\def\nsob#1#2{|#1|_{#2}}
\def\normbornva#1#2#3{\|#1\|_{L^\infty({#2};{#3})}}
\def\refer#1{~\ref{#1}}
\def\refeq#1{~(\ref{#1})}
\def\ccite#1{~\cite{#1}}
\def\pagerefer#1{page~\pageref{#1}}
\def\referloin#1{~\ref{#1} page~\pageref{#1}}
\def\refeqloin#1{~(\ref{#1}) page~\pageref{#1}}
\def\suite#1#2#3{(#1_{#2})_{#2\in {#3}}}
\def\ssuite#1#2#3{\hbox{suite}\ (#1_{#2})_{#2\in {#3}}}
\def\longformule#1#2{
\displaylines{
\qquad{#1}
\hfill\cr
\hfill {#2}
\qquad\cr
}
}
\def\inte#1{
\displaystyle\mathop{#1\kern0pt}^\circ
}
\def\sumetage#1#2{\sum_{\substack{{#1}\\{#2}}}}
\def\limetage#1#2{\lim_{\substack{{#1}\\{#2}}}}
\def\infetage#1#2{\inf_{\substack{{#1}\\{#2}}}}
\def\maxetage#1#2{\max_{\substack{{#1}\\{#2}}}}
\def\supetage#1#2{\sup_{\substack{{#1}\\{#2}}}}
\def\prodetage#1#2{\prod_{\substack{{#1}\\{#2}}}}

\def\convm#1{\mathop{\star}\limits_{#1}
}
\def\vect#1{
\overrightarrow{#1}
}
\def\Hd#1{{\mathcal H}^{d/2+1}_{1,{#1}}}

\def\derconv#1{\partial_t#1 + v\cdot\nabla #1}
\def\esptourb{\sigma + L^2(\R^2;\R^2)}
\def\tourb{tour\bil\-lon}



\newcommand{\beq}{\begin{equation}}
\newcommand{\eeq}{\end{equation}}
\newcommand{\ben}{\begin{eqnarray}}
\newcommand{\een}{\end{eqnarray}}
\newcommand{\beno}{\begin{eqnarray*}}
\newcommand{\eeno}{\end{eqnarray*}}

\def\cA{{\mathcal A}}
\def\cB{{\mathcal B}}
\def\cC{{\mathcal C}}
\def\cD{{\mathcal D}}
\def\cE{{\mathcal E}}
\def\cF{{\mathcal F}}
\def\cG{{\mathcal G}}
\def\cH{{\mathcal H}}
\def\cI{{\mathcal I}}
\def\cJ{{\mathcal J}}
\def\cK{{\mathcal K}}
\def\cL{{\mathcal L}}
\def\cM{{\mathcal M}}
\def\cN{{\mathcal N}}
\def\cO{{\mathcal O}}
\def\cP{{\mathcal P}}
\def\cQ{{\mathcal Q}}
\def\cR{{\mathcal R}}
\def\cS{{\mathcal S}}
\def\cT{{\mathcal T}}
\def\cU{{\mathcal U}}
\def\cV{{\mathcal V}}
\def\cW{{\mathcal W}}
\def\cX{{\mathcal X}}
\def\cY{{\mathcal Y}}
\def\cZ{{\mathcal Z}}

\def\virgp{\raise 2pt\hbox{,}}
\def\cdotpv{\raise 2pt\hbox{;}}
\def\eqdef{\buildrel\hbox{\footnotesize d\'ef}\over =}
\def\eqdefa{\buildrel\hbox{\footnotesize def}\over =}
\def\Id{\mathop{\rm Id}\nolimits}
\def\limf{\mathop{\rm limf}\limits}
\def\limfst{\mathop{\rm limf\star}\limits}
\def\sgn{\mathop{\rm sgn}\nolimits}
\def\RE{\mathop{\Re e}\nolimits}
\def\IM{\mathop{\Im m}\nolimits}
\def\im {\mathop{\rm Im}\nolimits}
\def\Sp{\mathop{\rm Sp}\nolimits}
\def\C{\mathop{\mathbb C\kern 0pt}\nolimits}
\def\DD{\mathop{\mathbb D\kern 0pt}\nolimits}
\def\EE{\mathop{\mathbb E\kern 0pt}\nolimits}
\def\K{\mathop{\mathbb K\kern 0pt}\nolimits}
\def\N{\mathop{\mathbb  N\kern 0pt}\nolimits}
\def\Q{\mathop{\mathbb  Q\kern 0pt}\nolimits}
\def\R{{\mathop{\mathbb R\kern 0pt}\nolimits}}
\def\SS{\mathop{\mathbb  S\kern 0pt}\nolimits}
\def\St{\mathop{\mathbb  S\kern 0pt}\nolimits}
\def\Z{\mathop{\mathbb  Z\kern 0pt}\nolimits}
\def\ZZ{{\mathop{\mathbb  Z\kern 0pt}\nolimits}}
\def\H{{\mathop{{\mathbb  H\kern 0pt}}\nolimits}}
\def\PP{\mathop{\mathbb P\kern 0pt}\nolimits}
\def\TT{\mathop{\mathbb T\kern 0pt}\nolimits}
 \def\L {{\rm L}}

\def\h {{\rm h}}
\def\v {{\rm v}}

\newcommand{\ds}{\displaystyle}
\newcommand{\la}{\lambda}
\newcommand{\hn}{{\bf H}^n}
\newcommand{\hnn}{{\mathbf H}^{n'}}
\newcommand{\ulzs}{u^\lam_{z,s}}
\def\bes#1#2#3{{B^{#1}_{#2,#3}}}
\def\pbes#1#2#3{{\dot B^{#1}_{#2,#3}}}
\newcommand{\ppd}{\dot{\Delta}}
\def\psob#1{{\dot H^{#1}}}
\def\pc#1{{\dot C^{#1}}}
\newcommand{\Hl}{{{\mathcal  H}_\lam}}
\newcommand{\fal}{F_{\al, \lam}}
\newcommand{\Dh}{\Delta_{{\mathbf H}^n}}
\newcommand{\car}{{\mathbf 1}}
\newcommand{\X}{{\mathcal  X}}
\newcommand{\fgl}{F_{\g, \lam}}

\newcommand{\andf}{\quad\hbox{and}\quad}
\newcommand{\with}{\quad\hbox{with}\quad}

\def\beginproof {\noindent {\it Proof. }}
\def\endproof {\hfill $\Box$}


\def\vp{{\underline v}}
\def\presspO{{{\underline p}_0}}
\def\presspun{{{\underline p}_1}}
\def\wp{{\underline w}}
\def\wpe{{\underline w}^{\e}}
\def\vapp{v_{app}^\e}
\def\vapph{v_{app}^{\e,h}}
\def\vbar{\overline v}
\def\barEE{\underline{\EE}}


\def\demo{d\'e\-mons\-tra\-tion}
\def\dive{\mathop{\rm div}\nolimits}
\def\curl{\mathop{\rm curl}\nolimits}
\def\cdv{champ de vec\-teurs}
\def\cdvs{champs de vec\-teurs}
\def\cdvdivn{champ de vec\-teurs de diver\-gence nul\-le}
\def\cdvdivns{champs de vec\-teurs de diver\-gence
nul\-le}
\def\stp{stric\-te\-ment po\-si\-tif}
\def\stpe{stric\-te\-ment po\-si\-ti\-ve}
\def\reelnonentier{\R\setminus{\bf N}}
\def\qq{pour tout\ }
\def\qqe{pour toute\ }
\def\Supp{\mathop{\rm Supp}\nolimits\ }
\def\coinfty{in\-d\'e\-fi\-ni\-ment
dif\-f\'e\-ren\-tia\-ble \`a sup\-port com\-pact}
\def\coinftys{in\-d\'e\-fi\-ni\-ment
dif\-f\'e\-ren\-tia\-bles \`a sup\-port com\-pact}
\def\cinfty{in\-d\'e\-fi\-ni\-ment
dif\-f\'e\-ren\-tia\-ble}
\def\opd{op\'e\-ra\-teur pseu\-do-dif\-f\'e\-ren\tiel}
\def\opds{op\'e\-ra\-teurs pseu\-do-dif\-f\'e\-ren\-tiels}
\def\edps{\'equa\-tions aux d\'e\-ri\-v\'ees
par\-tiel\-les}
\def\edp{\'equa\-tion aux d\'e\-ri\-v\'ees
par\-tiel\-les}
\def\edpnl{\'equa\-tion aux d\'e\-ri\-v\'ees
par\-tiel\-les non li\-n\'e\-ai\-re}
\def\edpnls{\'equa\-tions aux d\'e\-ri\-v\'ees
par\-tiel\-les non li\-n\'e\-ai\-res}
\def\ets{espace topologique s\'epar\'e}
\def\ssi{si et seulement si}

\def\pde{partial differential equation}
\def\iff{if and only if}
\def\stpa{strictly positive}
\def\ode{ordinary differential equation}
\def\coinftya{compactly supported smooth}


  \title
{ The Calder\'on Problem in the $L^p$ Framework on Riemann Surfaces
  }

\numberwithin{equation}{section}

\author[Y. Ma]{Yilin Ma}
\address[Y. Ma]%
{F07-CARSLAW BUILDING, THE UNIVERSITY OF SYDNEY, CHIPPENDALE NSW, AUSTRALIA}
\email{K.Ma@maths.usyd.edu.au}

\begin{abstract}
The purpose of this article is to extend the uniqueness results for the two dimensional Calder\'on problem to unbounded potentials on general geometric settings. We prove that the Cauchy data sets for Schr\"odinger equations uniquely determines potentials in $L^{p}$ for $p> 4/3$. In doing so, we first recover singularities of the potential, from which point a $L^2$-based method of stationary phase can be applied. Both steps are done via constructions of complex geometric optic solutions and Carleman estimates.
\end{abstract}

\maketitle


\setcounter{equation}{0}

\section{Introduction}
Let $( M_0, g )$ be a compact Riemannian manifold with smooth boundary $\partial M_0$ and dimension $n \geq 2$. Suppose that $V$ is a function in $L^{p}(M_0)$ for $p > 1$, and that $0$ is not a Dirichlet eigenvalue of $\Delta_{g}+ V$, then the famous Calder\'on problem for the Schr\"odinger equation
\begin{alignat}{2} \label{schoequation}
\begin{split}
\begin{cases}
(\Delta_{g} + V)u = 0 &\ \text{in} \ M_0, \\
\hspace*{15.75mm} u = f  &\ \text{on} \ \partial M_0
\end{cases}
\end{split}
\end{alignat}
asks whether or not the Dirichlet-Neumann map
\begin{alignat*}{2}
\begin{split}
\Lambda:
\begin{cases}
H^{1/2}(\partial M_0) \rightarrow H^{-1/2}( \partial M_0), \\ 
\hspace*{16.5mm} f \mapsto {\partial_{\nu}u_{f}}_{|_{\partial M_0}} 
\end{cases}
\end{split}
\end{alignat*}
uniquely determines the potential $V$, where $\nu$ is the outward pointing unit normal vector field on $\partial M_0$ and $u_f$ solves (\ref{schoequation}) with Dirichlet condition $f$. If $M_0 = \Omega$ is a bounded domain in $\mathbb{R}^{n}$ with the Euclidean metric, $n \geq 3$ and $V \in \mathcal{C}^{\infty}(\overline{\Omega})$, then the pioneering breakthrough accomplished in \cite{calderonfirstresult} by J. Sylvester and G. Uhlmann using the method of Complex Geometric Optic (CGO) solutions gave the positive answer. Since them, a considerable number of results towards this direction have appeared in the literature, with the method of CGO solutions becoming a standard tool in the subject. For dimension $n \geq 3$, D. Dos Santos Ferreira, C.E. Kenig, M. Salo and G. Uhlmann in \cite{AMCalderon} solved the Calder\'on problem with smooth potentials on certain admissible Riemannian manifolds which have at least one Euclidean direction. In the $L^{p}$ framework, S. Chanillo solved the Calder\'on problem for $V \in L^{n/2}$ on Euclidean bounded domains in \cite{Ln/2Euclidean} and Ferreira-Kenig-Salo on admissible manifolds in \cite{Ln/2admissible}. One could also consider the much related {\it partial data} problem, where one makes measurement on a specific open subset $\Gamma \subset \partial M_0$ instead of the entire boundary. All of the results mentioned above have their partial data counterparts, see \cite{partialdataeuclidean,Ln/2partialdataeuclidean,Ln/2partialdatamanifold}.  \par  
It is known that the two dimensional case is formally determined and thus notably difficult, with the first uniqueness result by A. Nachman in \cite{nachmancalderon} for a bounded domain $\Omega \subset \mathbb{C}$. For a potential $V \in W^{2,p}(\Omega)$, $p > 2$, the successful implementation of CGO solutions was due to Bukgheim in \cite{bukcalderon}. This was improved to $V \in L^{p}(\Omega)$ for $p > 2$ in \cite{L2calderon} by E. Bl\aa sten, O.Y. Imanuvilov and M. Yamamoto, and later on E. Bl{\aa}sten, L. Tzou and J.-N. Wang in \cite{L4/3calderon} did the case of $V \in L^{p}(\Omega)$ for $p> 4/3$. Not as much is known if $M_0$ is a compact Riemann surface with smooth boundary. In this case, L. Tzou and C. Guillarou in \cite{leocalderon} extended Bukgheim's method to solve the Calder\'on problem for a potential $V \in W^{2,p}(M_0)$ with $p>2$ and partial data. However, their proof relied critically on the fact that $V$ is continuous. For the partial data problem in Euclidean geometry,  O.Y. Imanuvilov, G. Uhlmann and M. Yamamoto obtained this result for $V \in W^{2,p}(\Omega)$, $p >2$ in \cite{2dpartialdataold,2dpartialdatanew}.  \par 
To this day there has not been any work in establishing uniqueness for the Calder\'on problem on compact Riemann surfaces with smooth boundaries for unbounded potentials. Since the direct problem for the Schr\"odinger equation is well-posed for all $V \in L^{p}(M_0)$, $p>1$, it is reasonable to ask whether the inverse problem can be solved in this range as well. In this paper we take care of the cases for $p > 4/3$. It remains an interesting question to fill the gap for $p \in {]}1, 4/3{]} $. Our main result, formulated in terms of the graph of the Dirichlet-Neumann map, is thus the following:
\begin{theorem} \label{main theorem}
Let $(M_0,g)$ be a compact Riemann surface with smooth boundary $\partial M_0$. Assume that $V_1$ and $V_2$ are two complex valued functions in $L^{p}(M_0)$ for $p>4/3$, such that their corresponding Cauchy data sets
\begin{alignat*}{2}
\mathcal{C}_{j} \ \mydef \ \{ ( u_{|_{\partial M_0}}, \partial_{\nu}  u_{|_{\partial M_0}} ) \ / \ u \in H^{1}(&M_0),  \ (\Delta_{g} + V)u = 0 \} \\
& \subset H^{1/2}(\partial M_0) \times H^{-1/2}(\partial M_0), \ \ j = 1,2
\end{alignat*}
satisfy $\mathcal{C}_1 = \mathcal{C}_2$, then $V_1 = V_2$.
\end{theorem}
 \par 
\noindent To prove Theorem \ref{main theorem} we will extend the strategy of \cite{L4/3calderon} which follows the philosophy of Bukgheim. There it was important to know a priori that the difference $V_1 - V_2$ already lives in $L^{2}(M_0)$ in order to apply a generalised method of stationary phase. Hence we also prove the following
\begin{theorem} \label{improve regularity theorem}
Let $(M_0,g)$ be a compact Riemann surface with smooth boundary $\partial M_0$. Assume that $V_1$ and $V_2$ are two complex valued functions in $L^{p}(M_0)$ for $p>4/3$ such that $\mathcal{C}_1 = \mathcal{C}_{2}$, then $V_1 - V_2$ is in $L^{2}(M_0)$.
\end{theorem}
Prior to the work of Nachman, Z. Sun and G. Uhlmann in their work \cite{improveregularold} proved various versions of Theorem \ref{improve regularity theorem} on Euclidean geometry using the methods of higher dimensions. Later in \cite{improveregularnew}, V.S. Serov and L. P\"aiv\"arinta made improvements and established statements which are parallel to ours. In either cases, the methods of proof in these works relied on tools which had no obvious analogy for general geometries. Our proof of  Theorem \ref{improve regularity theorem} demonstrates how the idea of Bukgheim can be adapted to obtain the same result on a Riemann surface with smooth boundary. 
\section{Inhomogenous Cauchy-Riemann Problems}

In this section we discuss solutions of the inhomogenous Cauchy-Riemann equation $ \bar{\partial} u = f $ which will be important in the process of constructing Green's operators to various conjugated operators. 
\subsection{Riemann Surfaces}
We begin by establishing some notations on Riemann surfaces. If $M_0$ is a compact Riemann surface with smooth boundary, then we can identify it as the closure of a bounded subset contained in a larger compact Riemann surface $\overline{M}$ with interior $M$ and boundary $\partial M$. The Hodge star operator $\star$ acts on the cotangent bundle $T^{\star}M$ with eigenvalues $i, -i$ and their respective eigenspaces $T^{\star}_{1,0}M$ and $T^{\star}_{0,1}M$. In a holomorphic coordinate $z = x + i y$ one has $T^{\star}_{1,0}M = \mathbb{C} dz$ and $ T^{\star}_{0,1}M = \mathbb{C} d \bar{z} $ where $dz = dx + i dy$ and $d \bar{z} = dx - idy$ and the complexified cotangent bundle admits the splitting $\mathbb{C} T^{\star}M = T^{\star}_{1,0}M \oplus T^{\star}_{0,1}M$. This splitting induces the natural projections $\pi_{1,0}: \mathbb{C}T^{\star}M \rightarrow T^{\star}_{1,0}M$ and $\pi: \mathbb{C} T^{\star} M \rightarrow T^{\star}_{0,1}M$. We then define the Cauchy-Riemann operators as $\partial f  = \pi_{1,0} df$ and $ \bar{\partial} f  = \pi_{0,1} df $ for $ f \in \mathcal{C}^{\infty}(M)$. If we let $\Lambda^{k}M$ denote the real bundle of $k$-forms on $M$ and $\mathbb{C}\Lambda^{k}M$ be its complexfied bundle, then $\partial$ and $\bar{\partial}$ also extend to $\mathbb{C} \Lambda^{1}M \rightarrow \mathbb{C} \Lambda^{2}M$ by setting $\partial(\sigma_{1,0} + \sigma_{0,1} ) = d\sigma_{0,1}$ and $\bar{\partial}(\sigma_{1,0} + \sigma_{0,1}) = d\sigma_{1,0}$ if $\sigma_{1,0} \in T^{\star}_{1,0}M$ and $\sigma_{0,1} \in T^{\star}_{0,1}M$. They satisfy $d = \partial + \bar{\partial}$ and are expressed in local coordinates as $\partial f = \partial_{z} f dz  $ and $ \bar{\partial} f = \partial_{\bar{z}}f d\bar{z} $ for $ \partial_{z} = 2^{-1} (\partial_{x} - i \partial_{y})$ and $\partial_{\bar{z}} = 2^{-1} ( \partial_{x} + i \partial_{y} )$ on functions, as well as $\partial(udz + v d\bar{z}) = \partial_{z} v dz \wedge d \bar{z}$ and $\bar{\partial}(udz + v d\bar{z}) = \partial_{\bar{z}}u d\bar{z} \wedge dz$ for forms. The formal adjoints of $\partial$ and $\bar{\partial}$ are defined by $\partial^{\star} = - \star \bar{\partial} \star$ and $ \bar{\partial} = - \star \partial \star$. The Laplacian is defined by $\Delta_{g} f = 2 \bar{\partial}^{\star} \bar{\partial} = 2 \partial^{\star} \partial $. In local coordinates $z$ we write $dz d\bar{z}$ to denote the flat volume form $2^{-1} i dz \wedge d\bar{z} $. If $M$ has metric $g$ then we let $d\text{v}_{g}$ be the volume form of $M$ with respect to $g$.
\subsection{Inverting the $\bar{\partial}$ Operator} \label{invert dbar}
In this subsection we recall some facts regarding the inhomogenous $\bar{\partial}$ equations. For every $k \in \mathbb{N}$ and $p \in {]}1, \infty{[} $, it was proved in Proposition 2.3 of \cite{reflectionpaper} that there exists a bounded operator 
\begin{alignat}{2} \label{invert d-bar}
\bar{T}: W^{k,p}(M; T_{0,1}^{\star}M) \rightarrow W^{k+1,p}(M)
\end{alignat}
such that $\bar{\partial}\bar{T} = \text{Id} $.  By adapting to the definition of $\bar{\partial}^{\star}$, we can also extend (\ref{invert d-bar}) to obtain bounded right inverses of 
\begin{alignat}{2}
\bar{\partial}^{\star}: W^{k+1,p}(M; T^{\star}_{0,1}M) \rightarrow W^{k,p}(M) \ \ \text{and} \ \ \bar{\partial} : W^{k+1,p}(M ; T^{\star}_{1,0} M) \rightarrow W^{k,p}(M; \Lambda^{2} M) 
\end{alignat}
by setting respectively
\begin{alignat}{2}
\begin{split} \label{second order inverse dbar}
\bar{T}^{\star}: W^{k,p}(M) \rightarrow W^{k+1,p}(M; T_{0,1}^{\star}M): & \hspace*{2mm}  f \longmapsto 2 \bar{\partial} G f, \ \ \text{and} \\[1mm]
\bar{T}: W^{k,p}(M; \Lambda^2 M ) \rightarrow W^{k+1,p}(M; T^{\star}_{1,0}M): & \hspace*{2mm} \sigma \longmapsto - \star \hspace*{0.5mm} T^{\star} \star \sigma
\end{split}
\end{alignat}
where $G: W^{k,p}(M) \rightarrow W^{k+2, p}(M)$ by elliptic regularity is the Dirichlet Green's operator for the Laplacian on $M$. As it is the usual convention, we adopt the same notations $\bar{\partial}$ and $\bar{T}$ for mappings between forms of various orders since they will be obvious from the contexts. The operators $T$ and $T^{\star}$ are to be understood as the complex conjugates of $\bar{T}$ and $\bar{T}^{\star}$ respectively. \par 
Moreover, since $2 \bar{\partial}^{\star} \bar{\partial} G = \text{Id} $ and $\partial^{\star}: W^{k+1,p}(M; T^{\star}_{1,0}M) \rightarrow W^{k,p}(M)$ satisfies $\partial^{\star} =  - \star \bar{\partial} \hspace*{0.5mm} \star$ for $\bar{\partial} : W^{k+1,p}(M; T^{\star}_{1,0}M) \rightarrow W^{k,p}(M; \Lambda^{2}M)$, we have that $\bar{\partial}^{\star} \bar{T}^{\star} = \text{Id}$ and $ \bar{\partial} \bar{T} = \text{Id} $ in (\ref{second order inverse dbar}) as well. We remark that despite the notations, with our constructions the operators $\bar{T}$ and $\bar{T}^{\star}$ in general are not adjoints of one another. \par 
On a bounded domain $\Omega \subseteq \mathbb{C}$, the operator $\bar{T}$ in (\ref{invert d-bar}) has a well-known local analogy given by the operator $\bar{R}$, which is bounded on $W^{k,p}(\Omega) \rightarrow W^{k+1,p}(\Omega)$ and defined by 
\begin{alignat}{2} \label{local d-bar}
\bar{R}f \ \mydef \ \frac{1}{2\pi i}\int_{\Omega} f(\zeta) \hspace*{0.5mm} \frac{d\zeta \wedge d \bar{\zeta}}{\zeta - z},
\end{alignat} 
and in local coordinates (\ref{local d-bar}) solves $\partial_{\bar{z}} \bar{R} = \text{Id}$. It is obvious that $\bar{R}$ can naturally be extended to $W^{k,p}(\Omega;T^{\star}_{0,1}\Omega) \rightarrow W^{k+1, p}(\Omega)$ via identification of $W^{k,p}(\Omega)$ with $W^{k,p}(\Omega ; T^{\star}_{0,1}\Omega)$. As before, we let $R$ denote the complex conjugate of $\bar{R}$.
\par 
The following results were proved in Section 2 of \cite{leoconnection}.
\begin{lemma} \label{decomposition of inverse d-bar operator}
Let $\chi$ and $\chi'$ be smooth cut-off functions on $M$. Assume that $\chi$ is supported on a holomorphic chart $\Omega \subset M$ and $\chi'$ supported on a small open neighbourhood of $\Omega$ and identically $1$ on the support of $\chi$, then there exists integral operators 
\begin{alignat*}{2}
K: W^{k,p}(M; T_{0,1}^{\star}M) \rightarrow \mathcal{C}^{\infty}(M) \ \ \text{and} \ \ L: W^{k,p}(M) \rightarrow  \mathcal{C}^{\infty} (M; T^{\star}_{0,1}M)
\end{alignat*}
respectively with smooth kernels on $M^2$ such that 
\begin{alignat}{2} \label{dbar splitting}
\bar{T} (\chi \sigma) = \chi' \bar{R} (\chi \sigma) + K (\chi \sigma) \ \ \text{and} \ \ \bar{T}^{\star} (\chi f) =  \chi' \bar{R} (|g|^{1/2} \chi f) d\bar{z} + L (|g|^{1/2} \chi f)
\end{alignat}
for all $\sigma \in W^{k,p}(M; T^{\star}_{0,1}M)$ and $f \in W^{k,p}(M)$.
\end{lemma}
\noindent In particular, suppose that $\{ \Omega_j \}_{j \geq 0}$ is a finite collection of holomorphic charts in $M$. Let $\{ \chi_{j} \}_{j \geq 0}$ be a partition of unity subordinate to $\{ \Omega_j \}_{j \geq 0}$ such that for every $j \geq 0$, we can choose smooth cut-off function $\chi_{j}'$ that is identically $1$ on the support of $\chi_{j}$. We have by (\ref{dbar splitting}) that
\begin{alignat}{2} \label{global splitting}
\bar{T}\sigma = \sum_{j \geq 0} \chi_j' \bar{R} (\chi_j \sigma ) + K_{j} ( \chi_j \sigma) \ \ \text{and} \ \ \bar{T}^{\star} f = \sum_{j \geq 0} \chi'_{j} \bar{R} ( |g_j|^{1/2} \chi_j f) d\bar{z} + L_{j} (|g_{j}|^{1/2} \chi_j f)
\end{alignat}
for all $ \sigma \in W^{k,p}(M; T^{\star}_{0,1} M) $ and $f \in W^{k,p}(M)$ whose supports are contained in $\bigcup_{j \geq 0} \Omega_{j}$, where $\{ K_j \}_{j \geq 0}$ and $\{ L_j \}_{j \geq 0}$ are operators with smooth kernels.
\subsection{Special Holomorphic Amplitudes} \label{construction subsection special holomorphic amplitude}
To solve the inverse problem, we will also be required to consider special holomorphic functions as well as $1$-forms on $M$ parametrised by an auxiliary variable. \par 
Let us introduce some important constructions which will be used throughout this paper. Suppose that $\tilde{p}_0 $ is an arbitrary point in $M_0$. By the result in \cite{holofunction}, there exists a holomorphic function $\Phi : M \rightarrow \mathbb{C}$ which can be extended to a larger open Riemann surface $M'$ containing $\overline{M}$ with non-vanishing derivative on $M'$ and remains holomorphic there. Without loss of generality we may assume that $\Phi(\tilde{p}_0) = 0$. Thus by the inverse function theorem, for $r>0$ we can choose small neighbourhoods $\tilde{\Omega}_0 \subset \tilde{\Omega}_0' \subset M$ such that $\Phi: \tilde{\Omega}_0 \rightarrow D_{r}$ and $\Phi: \tilde{\Omega}_{0}' \rightarrow D_{2r}$ are biholomorphic maps, where $D_{r} \subset \mathbb{C}$ is the disk with radius $r$ centred at the origin. \par 
By compactness we now find finitely many distinct points $ \{ \tilde{p}_{0}, \tilde{p}_{1}, ... \tilde{p}_{N} \} \subset M_0$ which form the preimage of $\Phi_{|_{M_0}}$ under $0$, then there exists subsets $\{ \tilde{\Omega}_{j'} \}_{1 \leq j' \leq N}, \{ \tilde{\Omega}_{j'}' \}_{1 \leq j' \leq N}$ in $M$ such that $\tilde{p}_{j} \in \tilde{\Omega}_{j'}$ and $\tilde{\Omega}_{j} \subset \tilde{\Omega}_{j'}'$. Moreover, we can choose them to be such that
\begin{alignat*}{2}
& \hspace*{0.85cm} \bigcup_{1 \leq j' \leq N} \tilde{\Omega}_{j'} \cap M_0 = \Phi^{-1}_{|_{M_0}}(D_{r}), \ \ \Phi: \tilde{\Omega}_{j'} \rightarrow D_{r} \ \text{is biholomorphic, and} \\
& \bigcup_{1 \leq j' \leq N} \tilde{\Omega}_{j'}' \cap M_0  = {\Phi}^{-1}_{|_{M_0}}(D_{2r}), \ \ \Phi: \tilde{\Omega}_{j'}' \rightarrow D_{2r} \ \text{is biholomorphic}, \ 0 \leq j' \leq N.
\end{alignat*}
Thus $\{ \tilde{\Omega}_{j'} \}_{0 \leq j' \leq N}, \{ \tilde{\Omega}_{j'}' \}_{0 \leq j' \leq N}$ define holomorphic charts on $M$. Let $\{ \tilde{\chi}_{j'}, \tilde{\chi}_{j'}' \}_{0 \leq j' \leq N}$ be smooth cut-off functions such that $ \tilde{\Omega}_{j'} \subset \text{Supp} \hspace*{0.5mm} \tilde{\chi}_{j'} $, $ \tilde{\Omega}_{j'}' \subset \text{Supp} \hspace*{0.5mm} \tilde{\chi}_{j'}'$ and $\tilde{\chi}_{j'}, \tilde{\chi}_{j'}'$ are identically $1$ on these sets. Without loss of generality, we can choose $\tilde{\chi}_{j'}'$ to be supported in a chart on which $\Phi$ remains biholomorphic and $\tilde{\chi}_{j'}'$ and is identically $1$ on the support of $\tilde{\chi}_{j'}$. In particular, for every $1 \leq j' \leq N$ we assume that $\text{Supp} \hspace*{0.5mm} \tilde{\chi}_{j'}'$ contains no other point $p \in M$ such that $\Phi(p) \in D_{2r}$.  
\par We now construct our special holomorphic amplitudes. For convenient we will often employ the notation $\Omega_0 = \tilde{\Omega}_0$.
\begin{lemma} \label{special amplitude}
Let $\tilde{p}_{0}$ be an arbitrary point in $M_0$ for which $\Phi(\tilde{p}_0)=0$. Suppose that $\Omega_0 \subset M$ is a neighbourhood of $\tilde{p}_0$ such that $\Phi: \Omega_{0} \rightarrow D_r$ is biholomorphic for some $r>0$. Let $\{ \tilde{\Omega}_{j'}', \tilde{\chi}_{j'}, \tilde{\chi}_{j'}'  \}_{0 \leq j' \leq N}$ be the holomorphic charts chosen above with coordinates $\{ z_{j'} \}_{0 \leq j' \leq N}$, then the followings are true:
\begin{enumerate}
\item[-] For every $p_0 \in \tilde{\Omega}_0'$ there exists a smooth function $\tilde{a}(p;p_0)$ in $p \in M$ such that the function defined by 
\begin{alignat}{2} \label{condition for a}
a(p;p_0) \ \mydef \ \tilde{\chi}_0'(p) + ( \Phi(p) - \Phi(p_0) ) \tilde{a}(p;p_0)
\end{alignat}
is holomorphic on $M \times \tilde{\Omega}_0'$. \vspace*{1mm}
\item[-] For every $p_0 \in \tilde{\Omega}_0'$ there exists smooth (1,0)-forms $\tilde{b}_{j'}(p;p_0)$ in $p \in M$ such that the forms defined by
\begin{alignat}{2}  \label{locla condition for b}
b_{j'}(p;p_0) \ \mydef \ \tilde{\chi}_{j'}'(p) d z_{j'} + ( \Phi(p) - \Phi(p_0) ) \tilde{b}_{j'}(p;p_0), \ \ 0 \leq j' \leq N
\end{alignat}
are holomorphic on $M$. Moreover, on every holomorphic chart $\Omega$ of $M$, the coefficients of $b_{j'}$ are holomorphic on $\Omega \times \tilde{\Omega}_0'$.
\end{enumerate}
\end{lemma}
\begin{proof}
For every $p_0 \in \tilde{\Omega}_0'$ we let $N(p_0)$ be the finite set of points $p \in M$ such that $\Phi(p) = \Phi(p_0)$ by compactness, then on $M \backslash N(p_0)$ we can carry out the following calculations 
\begin{alignat}{2}
\label{holomorphic amplitude first compute}
\bar{\partial}_{p} \Big( \frac{ \tilde{\chi}_0' (p)}{ \Phi(p) - \Phi(p_0) } \Big) = \frac{\bar{\partial}_{p} \tilde{\chi}_0' (p)}{\Phi(p) - \Phi(p_0)} \ \ \text{and} \ \ \bar{\partial}_{p} \Big( \frac{ \tilde{\chi}_{j'}'(p) d z_{j'} }{ \Phi(p) - \Phi(p_0) } \Big) = \frac{ \bar{\partial}_{p}( \tilde{\chi}_{j'}'(p) dz_{j'} ) }{ \Phi(p) - \Phi(p_0) } .
\end{alignat}
Suppose that $\Phi(p) = \Phi(p_0)$, then by construction we must have either $p \in \tilde{\Omega}_{j'}'$ for some $0 \leq j' \leq N$, or $p$ is contained in sets away from the supports of $\tilde{\chi}_{j'}'$ for all $j'$. Since $\partial_{\bar{z}}\tilde{\chi}_0', \partial_{\bar{z}} \tilde{\chi}_{j'}' $ vanish identically on all such sets, the terms in (\ref{holomorphic amplitude first compute}) extend respectively to a smooth holomorphic $1$-form and a $2$-form on $M$ via extension by zero. Thanks to the existence in (\ref{invert d-bar}) and (\ref{second order inverse dbar}), we can choose $\tilde{a}$ and $\tilde{b}$ by
\begin{alignat}{2} \label{definition of R}
\tilde{a}(p;p_0) \ \mydef \ \bar{T}_{p'} \bar{\partial}_{p'} \Big( \frac{ \tilde{\chi}_0'(p')}{\Phi(p') - \Phi(p_0)} \Big) \ \ \text{and} \ \  \tilde{b}_{j'}(p;p_0) \ \mydef \ \bar{T}_{p'} \bar{\partial}_{p'} \Big( \frac{ \tilde{\chi}_{j'}'(p') dz_{j'}'  }{ \Phi(p') - \Phi(p_0) } \Big).
\end{alignat}
By the smoothness of (\ref{holomorphic amplitude first compute}) and the boundedness of $\bar{T}$ and $T$ we deduce that both $\tilde{a}(\cdot ; p_0)$ and $\tilde{b}( \cdot ; p_0 )$ are smooth for every $p_0 \in \tilde{\Omega}_0'$. Therefore, the prescriptions
\begin{alignat}{2} \label{definition of a}
\begin{split}
& \hspace*{0.1cm} a(p;p_0)  =  ( \Phi(p) - \Phi(p_0) ) \Big( \frac{ \tilde{\chi}_0'(p)}{\Phi(p) - \Phi(p_0)} - \tilde{a}(p;p_0) \Big) = \tilde{\chi}_0'(p) -  ( \Phi(p) - \Phi(p_0) ) \tilde{a}(p;p_0), \ \ \text{and} \\
& b_{j'}(p;p_0)  = ( \Phi(p) - \Phi(p_0)  ) \Big( \frac{ \tilde{\chi}_{j'}'(p) dz_{j'}  }{ \Phi(p) - \Phi(p_0)  } - \tilde{b}_{j'}(p;p_0) \Big) = \tilde{\chi}_{j'}'(p) dz_{j'} - (\Phi(p) -\Phi(p_0) ) \tilde{b}_{j'}(p;p_0)
\end{split}
\end{alignat}
define respectively a holomorphic function and (1,0)-forms in $p \in M$ as required. \par 
It remains to show that $\tilde{a}$ and the coefficients of $\tilde{b}_{j'}$ are holomorphic in $p_0 \in \Omega_0$, but it suffices to observe directly from (\ref{definition of R}) combined with Lemma \ref{decomposition of inverse d-bar operator} that for some locally supported smooth functions $\tilde{\chi}''_0$ and $\tilde{\chi}''_{j'}$, we can write $\tilde{a}$ and $\tilde{b}$ as
\begin{alignat}{2}
\begin{split} \label{integral structure of a and b}
& \hspace*{1.6cm} \tilde{a}(p;p_0) = ( \tilde{\chi}_0''(p) \bar{R}_{p'} + K_{p'} )\Big(  \frac{\bar{\partial}_{p'} \tilde{\chi}_{0}(p') }{\Phi(p') - \Phi(p_0)}\Big), \ \ \text{and} \\ 
& \tilde{b}_{j'}(p;p_0) = - \star ( \tilde{\chi}_{j'}''(p) \bar{R}_{p'} + L_{p'} )\Big( \frac{ |g_{j'}|^{1/2} \star \bar{\partial}_{p'}( \tilde{\chi}_{j'}'(p') dz_{j'}' )  }{\Phi(p') - \Phi(p_0)} \Big), \ \ 0 \leq j' \leq N
\end{split}
\end{alignat}
where $K$ and $L$ are operators with smooth kernels. In particular, these are linear combinations of absolutely convergent integrals in the variable $p'$ since $\Phi(p') \neq \Phi(p_0)$ for any $p_0 \in \tilde{\Omega}_{0}'$ on the supports of their integrands, thus we can apply Lebesgue's differentiation theorem to conclude that $\tilde{a}$ and $\tilde{b}$ depend holomorphically on $p_0 \in \tilde{\Omega}_0'$. The claim now follows directly from the structures of $a$ and $b_{j'}$.
\end{proof}

A useful consequence of the holomorphic dependencies of $a$ and $b$ on $p_0$ is the following.
\begin{lemma} \label{Taylor expansion Lemma}
Let $a$ and $\{ b_{j'} \}_{0 \leq j' \leq N}$ be the holormophic functions and (1,0)-forms constructed in Lemma (\ref{special amplitude}). Suppose that $ \{ \Omega_j \}_{j \geq 0}$ is a holomorphic atlas of $M_0$ in $M$ such that $\Phi: \Omega_{j} \rightarrow \Phi(\Omega_j)$ is biholomorphic for each $j \geq 0$ and that $\Phi(\Omega_{j})$ is sufficiently small. If $\{ \chi_j \}_{j \geq 0}$ is a partition of unity subordinate to $\{ \Omega_{j} \}_{j \geq 0}$ with coordinate $z_{j}$ and let $b^{(j)}_{j'} $ be the coefficients of $b_{j'}$ in the coordinates $z_{j}$, then there exists Taylor expansion
\begin{alignat}{2}
\begin{split} \label{Taylor series of a and b}
& \hspace*{0.8cm}  a(p;p_0) = \sum_{\mu \geq 0}  a_{\mu}(p) z_0^{\mu} \ \ \text{and} \\
& \chi_{j}(p) b_{j'}^{(j)}(p;p_0)  = \sum_{\mu \geq 0} \chi_j(z) b_{j',\mu}^{(j)}(z_j) z_0^{\mu}
\end{split}
\end{alignat}
for every $z_0 \in D_r$ under the changes of variables $\Phi_{|_{\Omega_0}}(p) = z_0$ and $\Phi_{|_{\Omega_j}}(p) = z_j$, where $a_{\mu}$ is holomorphic in $M$ and $b_{j',\mu}^{(j)}$ is holomorphic in $\Omega_j$ for every $\mu \geq 0$ and $j,j' \geq 0$. Moreover, 
\begin{alignat}{2} \label{estimate for Taylor expansion}
\begin{split}
& \sup_{z_0 \in D_r} \sum_{\mu \geq 0} \| a_{\mu} \|_{L^{\infty}(M_0)} |z_0|^{\mu}  \leq \sum_{\mu \geq 0} \| a_{\mu} \|_{L^{\infty}(M_0)} r^{\mu} < \infty, \ \ \text{and} \\
& \hspace*{0.35cm} \sup_{z_0 \in D_r}  \sum_{\mu \geq 0}  \| b_{j',\mu}^{(j)} \|_{L^{\infty}(\Omega_j)} |z_0|^{\mu}  \leq  \sum_{\mu \geq 0} \| b_{j',\mu}^{(j)} \|_{L^{\infty}(\Omega_{j})} r^{\mu}  < \infty
\end{split}
\end{alignat}
for every $j,j' \geq 0$.
\end{lemma} 
\begin{proof}
By Lemma \ref{special amplitude},  the local expression of $a(p; \cdot)$ in the chart $\tilde{\Omega}_0'$ is holomorphic on $D_{2r}$ so that its radius of convergence is greater than $r$. Since $a_{\mu} = \partial_{\bar{z}_0}a(p; 0)/ \mu !$ we certainly have that $a_{\mu}$ is holomorphic on $M$. We can now cover $M_0$ by the finite collection of holomorphic charts $\{ \Omega_j, \chi_j \}_{j \geq 0}$ and assume that $\Phi(\Omega_j) = \tilde{D}_{r_{j}}$ is a disk of radius $r_{j}>0$ for every $j \geq 0$. In particular, we can assume that the local representation of $a$ is holomorphic on $\tilde{D}_{2r_j} \times D_{2r}$. \par We obviously have 
\begin{alignat*}{2}
\sup_{z_0 \in D_r} \sum_{\mu \geq 0} \| a_{\mu} \|_{L^{\infty}(M)} |z_0|^{\mu} \leq \sum_{j \geq 0} \sum_{\mu \geq 0} \| \chi_{j} a_{\mu} \|_{L^{\infty}(\Omega_j)} r^{\mu}.
\end{alignat*}
By standard power series theory, there exists constants $C_j > 0$ such that
\begin{alignat*}{2}
\sum_{\mu \geq 0} \| \chi_j a_{\mu} \|_{L^{\infty}(\Omega_j)} r^{\mu} \leq \sum_{\mu, \nu \geq 0} |a_{\mu,\nu}^{(j)}| r^{\mu} r_j^{\nu} \leq C_j < \infty
\end{alignat*}
where $a_{\mu,\nu}^{(j)}$ are the Taylor coefficients of $a_{\mu}$ in the chart $\Omega_j$. Summing over $j$ proves the first part of (\ref{estimate for Taylor expansion}). The case of $b_{j'}$ proceed similarly.
\end{proof}
\subsection{Inverting the Conjugated Operators}
Suppose that $\Phi$ is the holomorphic function without critical point chosen in subsection 2.3 and let $\varphi$, $\phi$ be respectively the real and imaginary parts of $\Phi$. For every $p_0 \in M$, we set $\Psi(p;p_0) = (\Phi(p) - \Phi(p_0) )^2$ and let $\psi$ be the real part of $\Psi$. Assume also that $\omega$ is a complex vector in $\mathbb{C}$ such that $|\omega|, \lambda > 0$. In this subsection we establish the conjugated operators that will be important in constructing special solutions to the Schr\"odinger equations. We first consider the conjugated Laplacians 
\begin{alignat*}{2}
P_{\Psi} \ \mydef \ e^{-i \Psi \lambda} \Delta_g e^{i \Psi \lambda}, \ P_{\bar{\Psi}} \ \mydef \ e^{-i \bar{\Psi} \lambda} \Delta_{g} e^{i \bar{\Psi} \lambda}, \ P_{\Phi} \ \mydef \ e^{ \frac{i}{2} \Phi \bar{\omega} } \Delta_{g} e^{- \frac{i}{2} \Phi \bar{\omega}}, \ P_{\bar{\Phi}} \ \mydef \ e^{\frac{i}{2} \bar{\Phi} \omega } \Delta_{g} e^{- \frac{i}{2} \bar{\Phi} \omega}.
\end{alignat*}
In the same way we may look at conjugated inverses to the Cauchy-Riemann operators
\begin{alignat*}{2}
&T_{\Psi} \ \mydef \ \frac{1}{2} e^{-2i \psi \lambda} T e^{2i \psi \lambda}, \hspace*{1.4mm} \ \ \bar{T}_{\Psi} \ \mydef \ \frac{1}{2} e^{-2i \psi \lambda} \bar{T} e^{2i \psi \lambda}, \ \ \text{and} \\
& \hspace*{0.55cm} T_{\Phi} \ \mydef \ \frac{1}{2} e^{ i \Phi \cdot \omega} T e^{- i \Phi \cdot \omega}, \hspace*{0.23cm} \ \ \bar{T}_{\Phi} \ \mydef \ \frac{1}{2} e^{ i \Phi \cdot \omega} \bar{T} e^{- i \Phi \cdot \omega}.
\end{alignat*}
Let $M_0'$ be a bounded domain in $M$ with smooth boundary such that $ M_0 \subset \overline{M_0'} \subset M $. We can choose $\tilde{\rho} \in \mathcal{C}^{\infty}_{c}(M_0')$ such that $\tilde{\rho}$ is identically $1$ near $M_0$. In view of the factorisation $\Delta_{g} = 2 \bar{\partial}^{\star} \bar{\partial} = 2 \partial^{\star} \partial$, it is clear that 
\begin{alignat*}{2}
P_{\Psi} T_{\Psi} \tilde{\rho} T^{\star} = P_{\bar{\Psi}} \bar{T}_{\Psi} \tilde{\rho} \bar{T}^{\star} = \text{Id} \ \ \text{and} \ \ P_{\Phi} T_{\Phi} \tilde{\rho} T^{\star} = P_{\bar{\Phi}} \bar{T}_{\Phi} \tilde{\rho} \bar{T}^{\star} = \text{Id}
\end{alignat*} 
on $W^{k,p}(M_0)$. 
\section{Carleman Estimates and Complex Geometric Optic Solutions}
A standard procedures in solving the Calder\'on problem is based on the Alessandrini's integral identity. For two potentials $V_{1}$ and $V_2$ in $L^{p}(M_0)$ and $p>1$, it is well known that if $\mathcal{C}_{1} = \mathcal{C}_{2}$, then 
\begin{alignat}{2} \label{Alessndrini identity}
\int_{M_0} u V v \hspace*{0.5mm} d\text{v}_{g} = 0, \ \ V \ \mydef \ V_1 - V_2 
\end{alignat}
for all $H^1(M_0)$ solutions $ u \in \text{Ker} \hspace*{0.5mm} (\Delta_g + V_1) $ and $v \in \text{Ker} \hspace*{0.5mm} (\Delta_{g} + V_2)$. Thus the orthogonality (\ref{Alessndrini identity}) implies that the identification of $V_1 = V_2$ hinges on the extend to which we can find such solutions. In this section we construct various specials CGO solutions to the Schr\"odinger equation, which will enable us to solve the Calder\'on problem from (\ref{Alessndrini identity}). \par 
Let $a$ be the holomorphic function constructed in subsection \ref{construction subsection special holomorphic amplitude}. The CGO solutions we look for will be of the forms
\begin{alignat*}{2}
u = e^{i \Psi \lambda} (a + r), \ \ v = e^{i \bar{\Psi} \lambda} ( \bar{a} + s), \ \  \tilde{u} = e^{- \frac{i}{2} \Phi \bar{\omega}} ( a + \tilde{r}), \ \ \tilde{v} = e^{- \frac{i}{2} \bar{\Phi} \omega} ( \bar{a} + \tilde{s} ) 
\end{alignat*}
for sufficiently large $\lambda$ and $\omega$. Since $e^{i \Psi \lambda} a$ and $e^{-\frac{i}{2} \Phi \bar{\omega}} a$ are holomorphic, $e^{i \bar{\Psi} \lambda} \bar{a}$ and $e^{-\frac{i}{2} \bar{\Phi} \omega} \bar{a} $ are antiholomorphic, they must also be harmonics. Thus the conditions we impose on the remainders should be  
\begin{alignat}{2}  \label{condition for solution}
P_{\Psi}r = - V_{1}(a+r), \ \ P_{\bar{\Psi}}s = - V_2 (a+s), \ \ P_{\Phi} \tilde{r} = - V_{2} ( a + \tilde{r}), \ \ P_{\bar{\Phi}} \tilde{s} = - V_{2} ( a + \tilde{s}).
\end{alignat}
Corresponding to any $\tilde{p}_0 \in M_0$ we recall the construction of holomorphic $(1,0)$-forms $\{ b_{j'} \}_{1 \leq j' \leq N}$ in Lemma \ref{special amplitude}. Let $\{ Q_{j', \epsilon}^{+} \}_{0 \leq j' \leq N}$ and $\{ Q_{j', \epsilon}^{-} \}_{0 \leq j' \leq N}$ be sequences of $\mathcal{C}^{\infty}_{c}(M)$ functions which will be chosen depending on $\epsilon >0$ later. We set
\begin{alignat*}{2}
u_1 \ &\mydef \ T_{\Psi} \tilde{\rho} \Big( T^{\star} V_{1}a - \sum_{1 \leq j' \leq N} Q_{j', \epsilon}^{+}(p_0)b_{j'} \Big), \ \ \tilde{u}_1 \ \mydef \ T_{\Phi} \tilde{\rho} T^{\star} V_1 a  , \\
v_{1} \ &\mydef \ \bar{T}_{\Psi} \tilde{\rho} \Big( \bar{T}^{\star} V_2 \bar{a} - \sum_{1 \leq  j' \leq N } Q_{j', \epsilon}^{-}(p_0)\bar{b}_{j'} \Big), \hspace*{2.8mm} \tilde{v}_{1} \ \mydef \ \bar{T}_{\Phi} \tilde{\rho} \bar{T}^{\star} V_2 \bar{a}  .
\end{alignat*}
One then easily observes from (\ref{condition for solution}) and a direct computation that if 
\begin{alignat*}{2}
& u_{j} \ \mydef \ T_{\Psi} \tilde{\rho} T^{\star}(V_1 u_{j-1}), \ \ \tilde{u}_{j} \ \mydef \ T_{\Phi} \tilde{\rho} T^{\star}(V_1 \tilde{u}_{j-1}), \\
& \hspace*{0.7mm} v_{j} \ \mydef\ \bar{T}_{\Psi} \tilde{\rho} \bar{T}^{\star}(V_2 v_{j-1}), \ \ \hspace*{0.5mm} \tilde{v}_{j} \ \mydef \ \bar{T}_{\Phi} \tilde{\rho} \bar{T}^{\star}(V_2 \tilde{v}_{j-1})
\end{alignat*}
for all $j \geq 0$, then the functions defined by 
\begin{alignat}{2} \label{choice of remainder}
r \ \mydef \ \sum_{j \geq 1}  (-1)^{j} u_{j},  \ \ \tilde{r} \ \mydef \ \sum_{j \geq 1}  (-1)^{j} \tilde{u}_{j}, \ \ s \ \mydef \ \sum_{j \geq 1}  (-1)^{j} v_{j}, \ \ \tilde{s} \ \mydef \ \sum_{j \geq 1}  (-1)^{j} \tilde{v}_j
\end{alignat}
will formally satisfy conditions (\ref{condition for solution}). We complement these definitions by setting $u_0 = \tilde{u}_0 = a$ and $v_0 = \tilde{v}_0 = \bar{a}$. We elaborate on these findings in the followings.
\begin{proposition} \label{existence of CGO solutions}
Let $V_1, V_2$ be in $L^{p}(M_0)$ for $ p > 4/3 $, then there exists $\lambda_0 > 0$ and $\omega_0 \in \mathbb{C}$ such that for all $\lambda > \lambda_0$ and $|\omega| > |\omega_0|$, we can find functions $u, v, \tilde{u}, \tilde{v}$ in $L^{\infty}(M_0 ) \cap H^{1}(M_0)$ of the forms
\begin{alignat}{2} \label{series expansion solutions}
\begin{split}
u &= \sum_{j \geq 0} (-1)^{j} e^{i \Psi \lambda} u_j, \ \  \tilde{u} = \sum_{j\geq 0} (-1)^{j} e^{- \frac{i}{2} \Phi \bar{\omega}} \tilde{u}_j, \\
v &= \sum_{j \geq 0} (-1)^{j} e^{i \bar{\Psi} \lambda} u_{j}, \ \ \tilde{v} = \sum_{j \geq 0} (-1)^{j} e ^{- \frac{i}{2} \bar{\Phi} \omega} \tilde{v}_{j},
\end{split}
\end{alignat}
with $u, \tilde{u}$ belonging to $\text{Ker} \hspace*{1mm} (\Delta_{g} + V_1)$ and $v, \tilde{v}$  belonging to $\text{Ker} \hspace*{1mm} (\Delta_g + V_2)$.
\end{proposition}
\begin{proof} 
We have shown that the functions defined by (\ref{series expansion solutions}) satisfy formally their corresponding Schr\"odinger equations. Thus it suffices to show that the series converge in appropriate spaces. We do so by first showing that they satisfy sufficiently nice asymptotic behaviour.  These will be done via the following estimates. \par 
Recall the surface $M_0'$ defined in Subsection 2.4.
\begin{proposition} \label{Carleman Estimate}
Assume that $(p, q, r)$ is in ${]}4/3, 2 {[} \times {]}4, \infty{[} \times {]}2, \infty{}[$ and $1/2 + 1/q \geq 1/p > 1/2$, then there exists a constant $C>0$ independent of $p_0 \in \Omega_0$ such that
\begin{alignat}{2} 
\label{dbar estimate}
\| \bar{T}_{\Psi} \sigma \|_{L^{q}(M_0)} & \leq \frac{C \| \sigma \|_{W^{1,p}(M;T^{\star}_{0,1}M)}}{\lambda^{1- (\frac{1}{p}-\frac{1}{q})}}, \ \ \| \bar{T}_{\Psi} \sigma \|_{L^{\infty}(M_0)} \leq \frac{C \| \sigma \|_{W^{1,r}(M;T^{\star}_{0,1}M)} }{\lambda^{\frac{1}{r}}}
\end{alignat}
for all $\sigma \in W^{1,p}_{0} (M_0';T^{\star}_{0,1}M_0')$. Alternatively, if $(p,q)$ is in ${]}4/3, 2{[} \times {]}2,4{[}$ and $1/p + 1/q = 1$, then we have
\begin{alignat}{2} \label{dbar estimate2}
\| \bar{T}_{\Phi} \sigma \|_{L^{q}(M_0)} & \leq \frac{ C \| \sigma \|_{W^{1,p}(M;T^{\star}_{0,1}M)}  }{|\omega|}, \ \ \| \bar{T}_{\Phi} \sigma \|_{L^{\infty}(M_0)} \leq \frac{ C\| \sigma \|_{W^{1,r}(M;T^{\star}_{0,1}M)}}{|\omega|}
\end{alignat}
for all $\sigma \in W^{1,p}_{0} (M_0';T^{\star}_{0,1}M_0')$.
\end{proposition}
By equation (\ref{global splitting}), it is sufficient to prove (\ref{dbar estimate}) and (\ref{dbar estimate2}) locally for $\bar{R}$ and globally for an operator $K$ with smooth kernel. Moreover, by density we may assume that $\sigma \in \mathcal{C}^{\infty}_{c}(M'_0; T^{\star}_{0,1}M_0')$, and so by partition of unity and the fact that $\Phi$ has non-vanishing derivative on $M$, it is enough to assume that the support of $\sigma$ is compactly contained in a holomorphic chart $\Omega \subset M$ on which $\Phi_{|_{\Omega}}(p) = z$. \par 
Let us first consider the local case, the proof of which we partially recall from \cite{L4/3calderon}. 
\begin{lemma} \label{Carleman estimate local lemma}
Assume that $(p,q,r)$ is in ${]}4/3, 2{[} \times {]}4, \infty{[} \times {]}2, \infty{[}$ and $1/2 + 1/q \geq 1/p > 1/2$, then there exists a constant $C>0$ independent of $p_0 \in \Omega_0$ such that
\begin{alignat}{2} \label{Local estimate}
 \| \bar{R} e^{2i \text{Re} \hspace{0.5mm} (z-z_0)^2 \lambda} f \|_{L^{q}(\Omega)} & \leq \frac{C \| f \|_{W^{1,p}(\Omega)}}{\lambda^{1- (\frac{1}{p}-\frac{1}{q}) }}, \ \ \| \bar{R} e^{2i \text{Re} \hspace*{0.5mm} (z-z_0)^2 \lambda} f \|_{L^{\infty}(\Omega)} \leq \frac{C \| f \|_{W^{1,r}(\Omega)}}{\lambda^{\frac{1}{r}}}
\end{alignat}
for all $f \in \mathcal{C}^{\infty}_{c}(\Omega)$. Alternatively, if $(p,q)$ is in ${]}4/3, 2{[} \times {]}2, 4{[}$ and $1/p + 1/q = 1$, then we have
\begin{alignat}{2}  \label{Local estimate2}
 \| \bar{R} e^{-i z \cdot \omega} f \|_{L^{q}(\Omega)} & \leq \frac{C \| f \|_{W^{1,p}(\Omega)}}{|\omega|}, \ \ \| \bar{R} e^{-i z \cdot \omega} f \|_{L^{\infty}(\Omega)} \leq \frac{C \| f \|_{W^{1,r}(\Omega)}}{|\omega|}
\end{alignat}
for all $f \in \mathcal{C}^{\infty}_{c}(\Omega)$.
\end{lemma}
\begin{proof}
Let $\chi$ in $\mathcal{C}^{\infty}_{c}(\mathbb{C})$ be identically $1$ for $|z| \geq 2 $ and vanishes for all $|z| \leq 1$. Setting $\chi_{\lambda} = \chi( \lambda^{1/2} (z-z_0))$, we can  write
\begin{alignat}{2} \label{local split}
\bar{R} e^{2i \text{Re} \hspace*{0.5mm} (z-z_0)^2 \lambda} f = \bar{R} e^{2i \text{Re} \hspace*{0.5mm} (z-z_0)^2 \lambda} \chi_{\lambda} f + \bar{R} e^{2i \text{Re} \hspace*{0.5mm (z-z_0)^2 \lambda}} (1- \chi_{\lambda} ) f.
\end{alignat}
Assume first $1/2 + 1/q \geq 1/p > 1/2$ and set $(p', q', r_1, r_2)$ to be such that
\begin{alignat}{2} \label{new indices}
\frac{1}{p'} \ \mydef \ \frac{1}{p} - \frac{1}{2}, \ \ \frac{1}{q'} \ \mydef \ \frac{1}{q} + \frac{1}{2} \ \ \text{and} \ \ \frac{1}{r_1} \ \mydef\ \frac{1}{q'} - \frac{1}{p'}, \ \ \frac{1}{r_2} \ \mydef \  \frac{1}{q} - \frac{1}{p'} = \frac{1}{q'} - \frac{1}{p}.
\end{alignat}
The final term in (\ref{local split}) can easily be estimated from Sobolev and H\"older's inequalities 
\begin{alignat}{2} \label{first split}
\begin{split}
\| \bar{R} & e^{2i \text{Re} \hspace*{0.5mm} (z-z_0)^2 \lambda} (1- \chi_{\lambda}) f \|_{L^{q}} \lesssim \| \bar{R} e^{2i \text{Re} (z-z_0)^2 \lambda} (1- \chi_{\lambda}) f \|_{W^{1,q'}} \\
& \hspace*{1cm} \lesssim \| (1- \chi_{\lambda})f \|_{L^{q'}} \lesssim \| 1-\chi_{\lambda} \|_{L^{r_1}} \| f \|_{L^{p'}} \lesssim \frac{\| f \|_{W^{1,p}}}{\lambda^{\frac{1}{r_1}}} = \frac{\| f \|_{W^{1,p}}}{\lambda^{1 - (\frac{1}{p} - \frac{1}{q})}},
\end{split}
\end{alignat}
and likewise we can get that
\begin{alignat}{2} \label{first split L inf estimate}
\begin{split}
\| \bar{R}  e^{2i \text{Re} \hspace*{0.5mm} (z-z_0)^2 \lambda} & (1- \chi_{\lambda}) f \|_{L^{\infty}} \lesssim \| \bar{R} e^{2i \text{Re} \hspace*{0.5mm} (z-z_0)^2 \lambda} (1-\chi_{\lambda}) f \|_{W^{1,r}} \\
& \lesssim \| (1- \chi_{\lambda}) f \|_{L^{r}} \leq \| 1- \chi_{\lambda} \|_{L^{r}} \| f \|_{L^{\infty}} \lesssim \frac{\| f \|_{W^{1,r}}}{\lambda^{\frac{1}{r}}}.
\end{split}
\end{alignat}
For the first term in the expansion (\ref{local split}), we integrate by parts to see that
\begin{alignat}{2} \label{integrate by parts}
\bar{R} e^{2i \text{Re} (z-z_0)^2 \lambda} \chi_{\lambda} f = \frac{i}{2\lambda} \Big( e^{2i \text{Re} \hspace*{0.5mm} (z-z_0) \lambda } \frac{\chi_{\lambda} f}{\bar{z}-\bar{z}_0}  -  \bar{R} \big( e^{2i \text{Re} \hspace*{0.5mm} (z-z_0)^2 \lambda} \partial_{\bar{z}} ( \frac{\chi_{\lambda} f}{\bar{z} - \bar{z}_0} )   \Big).
\end{alignat}
It follows again from boundedness of $\bar{R}$ that
\begin{alignat}{2} \label{second split}
\| \bar{R} e^{2i \text{Re} \hspace*{0.5mm} (z-z_0)^2 \lambda} \chi_{\lambda}  f\|_{L^{q}} \lesssim \frac{1}{\lambda} \Big( \big\| \frac{\chi_{\lambda} f}{ \bar{\cdot} - \bar{z}_0 } \big\|_{L^{q}} + \big\| f \partial_{\bar{z}} ( \frac{\chi_{\lambda}}{\bar{\cdot} - \bar{z}_0} )   \big\|_{L^{q'}} + \big\| \frac{\chi_{\lambda}\partial_{\bar{z}}f}{\bar{\cdot} - \bar{z}_0} \big\|_{L^{q'}}  \Big).
\end{alignat}
Since $\chi_{\lambda}$ is supported away from the set on which $z=z_0$, the first term on the right hand side in (\ref{second split}) can be estimated by
\begin{alignat*}{2}
\big\| \frac{\chi_{\lambda} f}{ \bar{\cdot} - \bar{z}_0 } \big\|_{L^{q}} \lesssim \big\| \frac{\chi_\lambda }{\bar{\cdot} - \bar{z}_0} \big\|_{L^{r_2}} \| f \|_{L^{p'}} \lesssim \frac{\| f \|_{W^{1,p}}}{\lambda^{\frac{1}{r_2}-\frac{1}{2}}} = \frac{\| f \|_{W^{1,p}}}{ \lambda^{ \frac{1}{q} - \frac{1}{p} } }
\end{alignat*}
since $r_2 > 2$. By the same reasonings, the second term satisfies
\begin{alignat*}{2}
\big\| f \partial_{\bar{z}}( & \frac{\chi_{\lambda}}{\bar{\cdot} - \bar{z}_0} ) \big\|_{L^{q'}}  \lesssim  \big\| \frac{ f \partial_{\bar{z}} \chi_{\lambda}}{\bar{\cdot} - \bar{z}_0} \big\|_{L^{q'}} + \big\| \frac{\chi_{\lambda} f}{ | \hspace*{0.5mm} \bar{\cdot} - \bar{z}_0 |^2}  \big\|_{L^{q'}} \\
& \lesssim \lambda^{1/2} \big\| \frac{(\partial_{\bar{z}} \chi)( \lambda^{1/2} \cdot )}{\bar{\cdot}} \big\|_{L^{r_1}} \| f \|_{L^{p'}} + \big\| \frac{\chi_{\lambda}}{ | \hspace*{0.5mm} \bar{\cdot} - \bar{z}_0|^2 } \big\|_{L^{r_1}} \| f \|_{L^{p'}} \lesssim \frac{\| f \|_{W^{1,p}}}{ \lambda^{\frac{1}{r_1} -1} } = \frac{\| f \|_{W^{1,p}}}{\lambda^{\frac{1}{q} - \frac{1}{p}}},
\end{alignat*}
and for the last term, we have
\begin{alignat*}{2}
\big\| \frac{\chi_{\lambda}\partial_{\bar{z}}f}{\bar{\cdot} - \bar{z}_0} \big\|_{L^{q'}} \lesssim \big\| \frac{\chi_{\lambda}}{\bar{\cdot} - \bar{z}_0}   \big\|_{L^{r_2}} \| \partial_{\bar{z}}f \|_{L^{p}} \lesssim \frac{\| f \|_{W^{1,p}}}{ \lambda^{\frac{1}{r_2} - \frac{1}{2}}} = \frac{\| f \|_{W^{1,p}}}{\lambda^{\frac{1}{q} - \frac{1}{p}}}.
\end{alignat*}
Putting everything back into (\ref{second split}) and combined with (\ref{first split}), we arrive at the first estimate in (\ref{Local estimate}). Applying the same strategy with respect to the supremum norm, we also have from (\ref{integrate by parts}) that
\begin{alignat}{2} 
\begin{split}
\| \bar{R} &  e^{2i \text{Re} \hspace*{0.5mm} (z-z_0)^2 \lambda} \chi_{\lambda} f \|_{L^{\infty}} \leq \frac{1}{\lambda} \Big( \big\| \frac{\chi_{\lambda} f}{\bar{\cdot} - \bar{z}_0} \big\|_{L^{\infty}} + \big\| f \partial_{\bar{z}}( \frac{\chi_{\lambda}}{\bar{\cdot} - \bar{z}_0} ) \big\|_{L^{r}} + \big\| \frac{\chi_{\lambda} \partial_{\bar{z}}f }{\bar{\cdot} - \bar{z}_0} \big\|_{L^{r}} \Big) \\
& \leq \frac{1}{\lambda} \Big( \big\| \frac{\chi_{\lambda} f}{\bar{\cdot} - \bar{z}_0} \big\|_{L^{\infty}} + \big\| \frac{f \partial_{\bar{z}} \chi_{\lambda}}{\bar{\cdot} - \bar{z}_0} \big\|_{L^{r}} + \big\| \frac{\chi_{\lambda}f}{ | \hspace*{0.5mm} \bar{\cdot} - \bar{z}_0 |^2 } \big\|_{L^r} + \big\| \frac{\chi_{\lambda} \partial_{\bar{z}}f}{\bar{\cdot} - \bar{z}_0} \big\|_{L^{r}} \Big) \lesssim \frac{\| f \|_{W^{1,r}}}{\lambda^{\frac{1}{r}}}.
\end{split}
\end{alignat}
Combining the above with inequality (\ref{first split L inf estimate}) yields the second estimate in (\ref{Local estimate}). \par 
To obtain (\ref{Local estimate2}), it is enough to note that in this case, the identity
\begin{alignat}{2} \label{better integrate by parts}
\bar{R} e^{ -i z \cdot \omega} f = \frac{2}{i \omega} \Big(  e^{-i z \cdot \omega }f  +  \bar{R} \big(  e^{-iz \cdot \omega} \partial_{\bar{z}} f  \big)  \Big).
\end{alignat}
holds conveniently without the need for localisation. Since $p' \geq q$ for all $p \geq 4/3$ and $1/p + 1/q = 1$, one can estimate directly using Sobolev's inequality to get that
\begin{alignat*}{2}
\| f \|_{L^{q}} \lesssim \| f \|_{L^{p'}} \lesssim \| f \|_{W^{1,p}} \ \ \text{and} \ \ \| \bar{R} e^{-iz \cdot \omega} \partial_{\bar{z}} f \|_{L^{q}} \lesssim \| \partial_{\bar{z}}f \|_{L^{p}} \leq \| f \|_{W^{1,p}}. 
\end{alignat*}
Combining the above with (\ref{better integrate by parts}) proves the first estimate in (\ref{Local estimate2}) and the second one follows from the same argument by applying the embedding $W^{1,r} \hookrightarrow L^{\infty} $. Notice also that after changing variables and taking modulus, the bounds we obtain are independent of $z_0$. Thus we have arrived at the required claims.
\end{proof}

We now move on to the smoothing terms in (\ref{global splitting}). We do not provide much details because the argument will be analogous to the proof of Lemma \ref{Carleman estimate local lemma}. Since we only consider $\sigma \in \mathcal{C}^{\infty}_{c}(M_0'; T^{\star}_{0,1}M_0')$ which are supported on a local chart $\Omega$, by identifying $\sigma$ with its coefficients in this chart, it is sufficient to consider an operator $K: W^{k,p}(\Omega) \rightarrow \mathcal{C}^{\infty}(M)$ with smooth kernel, thus we prove
\begin{lemma} \label{smooth kernel carleman estimate lemma}
Assume that $(p,q,r)$ is in ${]}4/3, 2{[} \times {]}4, \infty{[} \times {]}2, \infty{[}$ and $1/2 + 1/q \geq 1/p > 1/2$, then there exists a constant $C>0$ independent of $p_0 \in \Omega_0$ such that
\begin{alignat}{2} \label{global Carleman estimate}
\| K e^{2i \text{Re} \hspace*{0.5mm} (z-z_0)^2  \lambda} f  \|_{L^{q}(M_0)} & \leq \frac{C \| f \|_{W^{1,p}(\Omega)}}{\lambda^{1- (\frac{1}{p}-\frac{1}{q})}} , \ \ \| K e^{2i \text{Re} \hspace*{0.5mm} (z-z_0)^2 \lambda}  f\|_{L^{\infty}(M_0)} \leq \frac{C \| f \|_{W^{1,r}(\Omega)}}{\lambda^{\frac{1}{r}}}
\end{alignat}
for all $f \in \mathcal{C}^{\infty}_{c}(\Omega)$. Alternatively, if $(p,q)$ is in ${]}4/3, 2{[} \times {]}2, 4{[}$ and $1/p + 1/q = 1$, then we have
\begin{alignat}{2}
\| K e^{-i z \cdot \omega} f \|_{L^{q}(M_0)} & \leq \frac{C \| f \|_{W^{1,p}(\Omega)}}{|\omega|}, \ \ \| K e^{-iz \cdot \omega} f \|_{L^{\infty}(M_0)} \leq \frac{C \| f \|_{W^{1,r} (\Omega) }}{|\omega|}
\end{alignat} 
for all $f \in \mathcal{C}^{\infty}_{c}(\Omega)$.
\end{lemma}
\begin{proof}
Let $\chi_{\lambda}$ be the compactly supported function defined in the proof of Lemma \ref{Carleman estimate local lemma}. Since $K$ has smooth kernel, by Minkowski's inequality it is obvious that $K$ satisfies the same boundedness properties as $\bar{R}$. For $1/2 + 1/q \geq 1/p > 1/2$ we may split as before to get 
\begin{alignat}{2} \label{global split}
K e^{2i \text{Re} \hspace*{0.5mm} (z-z_0)^2 \lambda} f = K e^{2i \text{Re} \hspace*{0.5mm} (z-z_0)^2 \lambda} \chi_{\lambda} f + K e^{2i \text{Re} \hspace*{0.5mm} (z-z_0)^2 \lambda} (1- \chi_{\lambda}) f.
\end{alignat}
The local term in (\ref{global split}) clearly satisfies
\begin{alignat*}{2}
\| K e^{2i\text{Re} \hspace*{0.5mm} (z-z_0)^2 \lambda } (1- \chi_{\lambda}) f \|_{L^{q}} \lesssim \frac{\| f \|_{W^{1,p}}}{\lambda^{1 - (\frac{1}{p} - \frac{1}{q})}}.
\end{alignat*}
Integrating by parts, we also see that there exists an operator $K'$ with smooth kernel so that
\begin{alignat*}{2}
Ke^{2i \text{Re} \hspace*{0.5mm} (z-z_0)^2 \lambda} \chi_{\lambda} f = \frac{i}{2\lambda} K' e^{2i \text{Re} \hspace*{0.5mm} (z-z_0)^2 \lambda} \frac{\chi_{\lambda} f}{\bar{z} - \bar{z}_0}  + \frac{i}{2 \lambda} K e^{2i \text{Re} \hspace*{0.5mm} (z-z_0)^2 \lambda} \partial_{\bar{z}} \big( \frac{ \chi_{\lambda} f}{\bar{z} - \bar{z}_0} \big).
\end{alignat*}
This is analogous to (\ref{integrate by parts}), so the proof now proceeds exactly as in Lemma \ref{Carleman estimate local lemma}, and we have
\begin{alignat*}{2}
\| K e^{2i \text{Re} \hspace*{0.5mm}(z-z_0)^2 \lambda} \chi_{\lambda} f \|_{L^{q}} \lesssim \frac{\| f \|_{W^{1,p}}}{\lambda^{1- (\frac{1}{p} - \frac{1}{q}) }}
\end{alignat*}
as well. The other claims follow similarly.
\end{proof}
\begin{proof}[Proof of Proposition \ref{Carleman Estimate}]
By compactness and the fact that $\Phi$ has non-vanishing derivative on $M$, we can find a finite collection of holomorphic charts $\{ \Omega_j \}_{j \geq 0}$ in $M$ which covers $M_0'$, such that $\Phi_{|_{\Omega_j}}(p) = z_j$ defines holomorphic coordinates on $\Omega_j$ for each $j \geq 0$. Let $\{ \chi_{j} \}_{j \geq 0}$ be a partition of unity subordinate to $\{ \Omega_{j} \}_{j \geq 0}$ and choose $ \{ \chi_{j}' \}_{j \geq 0} $ so that for each $j \geq 0$, $\chi_{j}'$ is supported in a neighbourhood of $\Omega_{j}$ on which $\Phi$ remains biholomorphic, and $\chi_{j}'$ is identically $1$ on the support of $\chi_{j}$, then by (\ref{global splitting}) we have that
\begin{alignat*}{2}
& e^{-2i \psi \lambda} Te^{2i\psi \lambda} \sigma = \sum_{j \geq 0}  e^{-2i \psi \lambda} \chi_{j}' e^{2i \text{Re} \hspace*{0.5mm} (z-z_0)^2 \lambda} \chi_j  \sigma +  e^{-2i \psi \lambda} K_{j} e^{2i \text{Re} \hspace*{0.5mm} (z-z_0)^2 \lambda} \chi_j \sigma \ \ \text{and} \\
& \hspace*{1.8cm} e^{i \Phi \cdot \omega} Te^{-i \Phi \cdot \omega} \sigma = \sum_{j \geq 0} e^{iz \cdot \omega} \chi_{j}' R e^{-i z \cdot \omega} \chi_{j} \sigma +e^{iz \cdot \omega} K_{j} e^{-i z \cdot \omega} \chi_j \sigma
\end{alignat*}
for all $\sigma \in \mathcal{C}_{c}^{\infty} (M_0'; T^{\star}_{0,1}M_0')$. We now apply Lemma \ref{Carleman estimate local lemma} and Lemma \ref{smooth kernel carleman estimate lemma} to get the required claims by density.
\end{proof}

We also deduce from Proposition \ref{Carleman Estimate} the following Carleman estimates.
\begin{corollary} \label{Carleman estimate corollary}
Assume that $(p,q,r)$ is in ${]}4/3, 2{[} \times {]}4, \infty{[} \times {]}2, \infty{[}$ and $1/2 + 1/q \geq 1/p > 1/2$, then there exists a constant $C>0$ independent of $p_0 \in \Omega_0$ such that
\begin{alignat}{2}  \label{Carleman estimate for the Laplacian}
\| \bar{T}_{\Psi} \tilde{\rho} \bar{T}^{\star}f \|_{L^{q}(M_0)} & \leq \frac{C \| f \|_{L^{p}(M_0)}}{\lambda^{1 - (\frac{1}{p} - \frac{1}{q})}}, \ \ \| \bar{T}_{\Psi} \tilde{\rho} \bar{T}^{\star} f \|_{L^{\infty}(M_0)} \leq \frac{C \| f \|_{L^{p}(M_0)}}{\lambda^{0+}}
\end{alignat}
for all $f \in L^{p}(M_0)$. Alternatively, if $(p,q)$ is in ${]}4/3, 2{[} \times {]}2, 4{[}$ and $1/p + 1/q = 1$, then we have
\begin{alignat}{2}
\| \bar{T}_{\Phi} \tilde{\rho} \bar{T}^{\star}f \|_{L^{q}(M_0)} & \leq \frac{ C \| f \|_{L^{p}(M_0)}}{|\omega|}, \ \ \| \bar{T}_{\Phi} \tilde{\rho} \bar{T}^{\star}f \|_{L^{\infty}(M_0)} \leq \frac{C \| f \|_{L^{p}(M_0)}}{|\omega|^{0+}}
\end{alignat}
for all $f \in L^{p}(M_0)$.
\end{corollary}
Here we adopt the notation $\alpha +$ to denote $\alpha + \epsilon$ for some $\epsilon > 0$ whenever $\alpha$ is real.
\begin{proof}
Suppose first that $f \in \mathcal{C}^{\infty}_{c}(M_0)$, then $ \tilde{\rho} \bar{T}^{\star}f \in \mathcal{C}^{\infty}_{c}(M_0'; T^{\star}_{0,1}M_0') $, thus the proofs for the $L^{p} \rightarrow L^{q}$ estimates in this case are obvious from Proposition \ref{Carleman Estimate}. On the other hand, we set $(p_0,p_1)$ to be such that $1 < p_0 < 4/3< p < 2 < p_1 $ and let $p_0' \in {]}2, 4{[}$ be defined by $1/p_0' = 1/p_0 - 1/2$, then by Sobolev inequalities and the boundedness of $\bar{T}$ and $\bar{T}^{\star}$, we have
\begin{alignat}{2} \label{interpolate 0}
\begin{split}
\| & \bar{T}_{\Psi} \tilde{\rho} \bar{T}^{\star} f \|_{L^{\infty}(M_0)} \lesssim \| \bar{T}_{\Psi} \tilde{\rho} \bar{T}^{\star}f \|_{W^{1,p_{0}'}(M)} \\[1mm] & \lesssim \| \tilde{\rho} \bar{T}^{\star}f \|_{L^{p'_0}(M_0';T^{\star}_{0,1}M_0')} \lesssim \| \bar{T}^{\star}f \|_{W^{1,p_0}(M ; T^{\star}_{0,1} M)} \lesssim \| f \|_{L^{p_0}(M_0)}.
\end{split}
\end{alignat}
On the other hand, a direct application of (\ref{dbar estimate}) yields
\begin{alignat}{2} \label{interpolate 1}
\| \bar{T}_{\Psi} \tilde{\rho} \bar{T}^{\star}f \|_{L^{\infty}(M_0)} \lesssim \frac{\| f \|_{L^{p_1}(M_0)}}{\lambda^{\frac{1}{p_1}}}.
\end{alignat}
Interpolating between (\ref{interpolate 0}) and (\ref{interpolate 1}) implies the second estimate in (\ref{Carleman estimate for the Laplacian}). We can get the other one using the same strategy. The claim now follows from density.
\end{proof}

It remains to show that the sums in (\ref{choice of remainder}) indeed converge. By iterating the $L^{p} \rightarrow L^{\infty}$ estimates in Corollary \ref{Carleman estimate corollary}, we see that there exists a constant $C>0$ such that 
\begin{alignat*}{2} 
& \| u_{j} \|_{L^{\infty}(M_0)}  \leq \Big( \frac{ C \| V_1 \|_{L^p} }{|\lambda|^{0+}} \Big)^{j-1} \| u_1 \|_{L^{\infty}(M_0 )} \ \ \text{for all} \ j \geq 2, \ \text{and} \\
& \hspace*{0.4cm} \| \tilde{u}_{j} \|_{L^{\infty}(M_0)} \leq \Big( \frac{C \| V_1 \|_{L^p} }{|\omega|^{0+}} \Big)^{j} \| a \|_{L^{\infty}(M_0 \times \Omega_0)} \ \ \text{for all} \ j \geq 0.
\end{alignat*}
Inserting the above inequalities into (\ref{series expansion solutions}), we get
\begin{alignat}{2} \label{solution estimates}
\begin{split}
& \big\| \sum_{j \geq 0}  (-1)^{j} u_{j} \big\|_{L^{\infty}(M_0)}  \leq  \sum_{0 \leq j < 2} \| u_{j} \|_{L^{\infty}(M_0)}  + \sum_{j \geq 2} \Big( \frac{ C \| V_1 \|_{L^{p}} }{\lambda^{0+}} \Big)^{j} \|u_{1}\|_{L^{\infty}(M_0)}, \ \ \text{and} \\
& \hspace*{2cm} \big\| \sum_{j \geq 0} (-1)^{j} \tilde{u}_{j} \big\|_{L^{\infty}(M_0)} \leq \sum_{j \geq 0} \Big( \frac{C \| V_{1} \|_{L^{p}}  }{|\omega|^{0+}} \Big)^j \| a \|_{L^{\infty}(M_0 \times \Omega_0)}.
\end{split}
\end{alignat}
We can further bound the $L^{\infty}$ norm of $u_1$ by Sobolev embedding so that
\begin{alignat*}{2}
\| u_1 \|_{L^{\infty}} \lesssim \| V_1 \|_{L^{p}} \| a \|_{L^{\infty}(M_0 \times \Omega_0)} + \max_{1 \leq j' \leq N} \sup_{p_0 \in \Omega_0} | Q^{+}_{j', \epsilon}(p_0)| \| b_{j'} \|_{L^{\infty}(M_0; T^{\star}_{1,0}M_0)} < \infty
\end{alignat*}
which is finite and so is $\| u_0 \|_{L^{\infty}(M_0 \times \Omega_0)}$. Thus we can find $\lambda_0 > 0$ and $\omega_0 \in \mathbb{C}$ such that for all $ \lambda > \lambda_0 $ and $ |\omega| > |\omega_0| $, the right hand sides of (\ref{solution estimates}) converge by geometric series. Hence $u$ and $\tilde{u}$ are in $L^{\infty}(M_0)$ for all sufficiently large $\lambda$ and $|\omega|$. Since $T_{\Psi} \tilde{\rho} T^{\star}$ and $T_{\Phi} \tilde{\rho} T^{\star}$ are bounded $L^{p}(M_0) \rightarrow W^{2, p}(M_0)$, we can now write
\begin{alignat*}{2}
& u = e^{i \Psi \lambda} a + T_{\Psi} \tilde{\rho} T^{\star} e^{i\Psi \lambda} V_1 r \ \ \text{and} \ \ \tilde{u} = e^{- \frac{i}{2} \Phi \bar{\omega}} a + T_{\Phi} \tilde{\rho} T^{\star} e^{- \frac{i}{2} \Phi \bar{\omega}} V_1 \tilde{r} .
\end{alignat*}
The embedding $W^{2,p}(M_0) \hookrightarrow H^1(M_0)$ then yields that $u$ and $\tilde{u}$ are $H^{1}(M_0)$ functions. Similar calculations work for $v$ and $\tilde{v}$. This concludes the proof of Proposition \ref{existence of CGO solutions}.
\end{proof}
\section{Improving Regularity}
In this section we prove Theorem \ref{improve regularity theorem}. Observe by compactness it suffices to show that $V \in L^{2}_{loc}(M_0)$. Indeed, we will show that
\begin{proposition} \label{assist theorem 2}
Any point in $M_0$ admits an open neighbourhood $\Omega_0 \subset M$ such that $V \in L^{2}(\Omega_0)$.  
\end{proposition}
\begin{proof}
Implementing the $L^{\infty}(M_0) \cap H^{1}(M_0)$ solutions $\tilde{u}$ and $\tilde{v}$ from Proposition \ref{existence of CGO solutions} into identity (\ref{Alessndrini identity}), we have 
\begin{alignat}{2} \label{Improvement expansion}
0 = \int_{M_0} \tilde{u} V \tilde{v} \hspace*{0.5mm} d\text{v}_{g} = \int_{M_0} e^{-i \Phi \cdot \omega} |a|^2 V \hspace*{0.5mm} d\text{v}_{g} +   \sum_{k + k' \geq 1}  \int_{M_0}  e^{- i \Phi \cdot \omega} \tilde{u}_{k} V \tilde{v}_{k'} \hspace*{0.5mm} d\text{v}_{g} 
\end{alignat}
where we switched the sum and integral by boundedness. Let $\Phi: M \rightarrow \mathbb{C}$ be the holomorphic function without critical point chosen in Section 2 and $\tilde{p}_0$ be an arbitrary point in $M_0$. Assume without loss of generality that $\Phi(\tilde{p}_0) = 0$. Let $ \{ \tilde{\Omega}_{j'} \}_{1 \leq j' \leq N} $ be the holomorphic charts constructed at the beginning of Section 2.3, and $ \{ \Omega_j \}_{j \geq 0} $ be an open covering of $M_0'$ in $M$ such that $\Phi: \Omega_j \rightarrow \Phi(\Omega_j)$ is biholomorphic for each $j \geq 0$. Fix a partition of unity $\{ \chi_j \}_{j \geq 0}$ subordinate to $\{ \Omega_{j} \}_{j \geq 0}$ and choose $\{ \chi_{j}' \}_{j \geq 0}$ so that $\chi_{j}'$ is supported on a holomorphic chart with coordinate map $\Phi$ and is identically $1$ on the support of $\chi_j$. We can modify the definitions by setting $\Omega_0 = \tilde{\Omega}_0$. 
\par 
Choose $\omega_0 \in \mathbb{C}$ so that $|\omega_0|$ is sufficiently large. Let $\rho$ be be a smooth function which vanishes on the set $|\omega| \leq |\omega_0| $. For $\epsilon > 0$ we can multiply both sides of equation (\ref{Improvement expansion}) by the weight $\mathbf{1}_{\Omega_0}(p_0) \rho(\omega) e^{-\epsilon |\omega|^2} e^{i \Phi(p_0) \cdot \omega} $ and integrate to get 
\begin{alignat}{2}
\begin{split}
 \label{modifiedi improv equation}
\mathbf{1}_{\Omega_0} (p_0) & \int_{|\omega| > |\omega_0|} \rho(\omega) e^{-\epsilon|\omega|^2}  e^{i \Phi(p_0) \cdot \omega} \int_{M_0} e^{-i \Phi \cdot \omega} |a|^2 V \hspace*{0.5mm} d\text{v}_{g} \\
& = - \mathbf{1}_{\Omega_0}(p_0) \int_{|\omega| > |\omega_0|} \rho(\omega) e^{-\epsilon |\omega|^2} e^{i \Phi(p_0) \cdot \omega} \sum_{k + k' \geq 1} \int_{M_0} e^{i \Phi(p_0) \cdot \omega} \tilde{u}_{k} V \tilde{v}_{k'} \hspace*{0.5mm} d\text{v}_{g}.
\end{split}
\end{alignat}
We now want to take a limit as $\epsilon \rightarrow 0$ in order to apply a Fourier inversion argument. By our construction of $a$, we show that the left hand side of (\ref{modifiedi improv equation}) converges in $L^1(M)$ to $\mathbf{1}_{\Omega_0} V |g|^{1/2}$ where $|g|$ is the volume component of $g$ in $\Omega_0$, while the right hand side converges to a $L^2(\Omega_0)$ function. The extra complexity is that the amplitude $a$ depends nonlinearly on both $p$ and $p_0$ while $V$ might not even be continuous. We resolve this problem by relying on the Taylor expansion introduced in Lemma (\ref{Taylor expansion Lemma}) which reduces the problem to the usual Fourier inversion theorem. Nevertheless, a more careful computation is required.
\subsection{Analysis of Principle Terms}
We first consider the left hand side of (\ref{modifiedi improv equation}). By the change of variables $\Phi_{|_{\Omega_0}}(p) = z_0$, we can integrate over the $\omega$ variable to see that
\begin{alignat}{2} \label{first improv expansion}
\begin{split}
\mathbf{1}_{\Omega_0} (&p_0  )  \int_{|\omega| \geq |\omega_0|}  \rho(\omega) e^{-\epsilon |\omega|^2} e^{- i \Phi(p_0) \cdot \omega} \Big( \int_{M_0} e^{-i \Phi \cdot \omega} |a|^2 V \hspace*{0.5mm} d\text{v}_{g} \Big) d\omega d\bar{\omega} \\
& \hspace*{0.77cm} = \mathbf{1}_{D_{r}} (z_0) \int_{\mathbb{C}} e^{-\epsilon |\omega|^2} e^{- i z_0 \cdot \omega} \Big( \int_{M_0} e^{-i \Phi \cdot \omega } |a|^2 V \hspace*{0.5mm} d\text{v}_{g} \hspace*{0.5mm} \Big) d \omega  d \bar{\omega} \\
 &  \hspace*{0.8cm} - \mathbf{1}_{D_r}(z_0) \int_{|\omega| < |\omega_0|} (1-\rho)(\omega ) e^{-\epsilon |\omega|^2} e^{-i z_0 \cdot \omega} \Big( \int_{M_0} e^{-i \Phi \cdot \omega} |a|^2 V \hspace*{0.5mm} d\text{v}_g \Big) \hspace*{0.5mm} d\omega  d\bar{\omega}.
\end{split}
\end{alignat}
We first look at the second line of (\ref{first improv expansion}), from which we shall obtain local information about $V$. This is summarised in the following technical lemma. Notice that by extending $V_1,V_2$ to $M$ via zero, it is enough to show
\begin{lemma} \label{principle term calculation improve lemma}
We have that
\begin{alignat}{2} \label{principle term calculation improve}
\lim_{\epsilon \rightarrow 0} \mathbf{1}_{\Omega_0}(p_0) \int_{\mathbb{C}} e^{-\epsilon |\omega|^2} e^{i \Phi(p_0) \cdot \omega  } \Big( \int_{M} e^{-i \Phi \cdot \omega} |a|^2 V \hspace*{0.5mm} d\text{\emph v}_{g} \Big) \hspace*{0.5mm} d\omega  d\bar{\omega} = \mathbf{1}_{\Omega_0} V |g|^{1/2}.
\end{alignat}
in $L^1(M)$.
\end{lemma}
\begin{proof}
Set
\begin{alignat*}{2}
\tilde{M} \ \mydef \ M \big\backslash \big( \bigcup_{1 \leq j' \leq N} \tilde{\Omega}_{j'} \big) 
\end{alignat*}
and introduce the splitting
\begin{alignat}{2} \label{improv first expansion}
\begin{split}
\mathbf{1}_{\Omega_0}(p_0) \int_{\mathbb{C}} e^{-\epsilon |\omega|^2}  e^{i \Phi(p_0) \cdot \omega  } & \Big( \int_{M} e^{-i \Phi \cdot \omega} |a|^2 V \hspace*{0.5mm} d\text{v}_{g} \Big) \hspace*{0.5mm} d\omega  d\bar{\omega} \\[1mm]
= \mathbf{1}_{\Omega_0}(p_0) \sum_{1 \leq j' \leq N} \int_{\mathbb{C}} & e^{-\epsilon |\omega|^2} e^{i \Phi(p_0) \cdot \omega} \Big( \int_{\tilde{\Omega}_{j'}} e^{-i \Phi \cdot \omega} |a|^2 V \hspace*{0.5mm} d\text{v}_{g} \Big) \hspace*{0.5mm} d\omega  d \bar{\omega} \\
& + \mathbf{1}_{\Omega_0}(p_0) \sum_{j \geq 0} \int_{\mathbb{C}} e^{-\epsilon |\omega|^2} e^{i \Phi(p_0) \cdot \omega} \Big( \int_{\tilde{M}} \chi_{j} e^{-i \Phi \cdot \omega }|a|^2 V \hspace*{0.5mm} d\text{v}_{g} \Big) \hspace*{0.5mm} d\omega d\bar{\omega}.
\end{split}
\end{alignat}
On each $\tilde{\Omega}_j$ we make the change of variable so that $\Phi_{|_{\tilde{\Omega}_{j'}}}(p) = z_{j'}$ and write
\begin{alignat}{2} \label{recovery principle calculation improv}
\begin{split}
 \mathbf{1}_{\Omega_0}(p_0) & \int_{\mathbb{C}} e^{-\epsilon |\omega|^2} e^{i \Phi(p_0) \cdot \omega} \Big( \int_{\tilde{\Omega}_{j'}} e^{-i \Phi \cdot \omega} |a|^2 V \hspace*{0.5mm} d\text{v}_{g} \Big) \hspace*{0.5mm} d\omega d \bar{\omega} \\
& = \mathbf{1}_{D_r}(z_0)   \int_{\mathbb{C}} e^{-\epsilon |\omega|^2} e^{i z_0 \cdot \omega} \Big( \int_{D_r}  e^{-iz_{j'} \cdot \omega} | a( \Phi_{|_{\tilde{\Omega}_{j'}}}^{-1}(z_{j'}) ; \Phi_{|_{\tilde{\Omega}_{0}}}^{-1} (z_0) ) |^2  \\[1mm]
& \hspace*{6cm} \times V( \Phi_{|_{\tilde{\Omega}_{j'}}}^{-1}(z_{j'}) ) \hspace*{1mm} |g_{j'}|^{1/2} dz_{j'} d\bar{z}_{j'} \Big) \hspace*{0.5mm} d\omega  d\bar{\omega}.
\end{split}
\end{alignat}
We recall the Taylor expansion from Lemma \ref{Taylor expansion Lemma}
\begin{alignat}{2} \label{holomorphic expansion of a}
 a( p; \Phi_{|_{\tilde{\Omega}_{0}}}^{-1}(z_0) ) = \sum_{\mu \geq 0} a_{\mu} (p ) z_{0}^{\mu}, \ \ z_0 \in D_r
\end{alignat}
with holomorphic coefficients on $M$. One thus observe from (\ref{estimate for Taylor expansion}) and the dominated convergence theorem that (\ref{recovery principle calculation improv}) can be written as  
\begin{alignat}{2} \label{expanded integral equation}
\begin{split}
 \sum_{\mu, \mu' \geq 0} & \mathbf{1}_{D_r}(z_0)  z_{0}^{\mu} \bar{z}_{0}^{\mu'} \int_{\mathbb{C}} e^{-\epsilon |\omega|^2} e^{i z_0 \cdot \omega} \\
& \times \Big( \int_{D_r} e^{-iz_{j'} \cdot \omega} a_{\mu}( \Phi_{|_{\tilde{\Omega}_{j'}}}^{-1} (z_{j'}) )  \bar{a}_{\mu'}( \Phi_{|_{\tilde{\Omega}_{j'}}}^{-1} (z_{j'}) )  V( \Phi_{|_{\tilde{\Omega}_{j'}}}^{-1}(z_{j'}) ) \hspace*{0.5mm} |g_{j'}|^{1/2} dz_{j'} d\bar{z}_{j'} \Big) \hspace*{0.5mm} d\omega  d \bar{\omega},
\end{split}
\end{alignat}
and so it can be easily seen from (\ref{estimate for Taylor expansion}) combined with (\ref{expanded integral equation}) that
\begin{alignat}{2} \label{estimate fourier recover 1}
\begin{split}
\Big\| & \int_{\mathbb{C}}  e^{-\epsilon |\omega|^2}  e^{i z_0 \cdot \omega} \Big( \int_{D_r} e^{-i z_{j'} \cdot \omega}|a( \Phi_{|_{\tilde{\Omega}_{j'}}}^{-1}(z_{j'}); \Phi_{|_{\Omega_0}}^{-1}(z_0) )|^2 V(\Phi_{|_{\tilde{\Omega}_{j'}}}^{-1}(z_{j'})) \hspace*{0.5mm} |g_{j'}|^{1/2} dz_{j'} d\bar{z}_{j'} \Big) \hspace*{0.5mm} d\omega  d\bar{\omega} \\
& \hspace*{3.5cm} - \mathbf{1}_{D_r}(z_0) [ |a( \Phi_{|_{\tilde{\Omega}_{j'}}}^{-1}(z_{j'});  \Phi_{|_{\Omega_{0}}}^{-1}(z_0) )|^2  V(\Phi_{|_{\tilde{\Omega}_{j'}}}^{-1}(z_{j'})) |g_{j'}|^{1/2}]_{|_{z_{j'} = z_0}}  \Big\|_{L^1(D_r)} \\
& \lesssim \sum_{\mu, \mu' \geq 0} r^{\mu + \mu'} \| o_{L^1,\epsilon}^{(\mu,\mu')}(1) \|_{L^{1}(D_r)} \\
&  \lesssim \Big( \sum_{\mu, \mu' \geq 0} r^{\mu + \mu'} \| a_{\mu} \|_{L^{\infty}(M_0)} \| a_{\mu'} \|_{L^{\infty}(M_0)} \Big) \big( \epsilon^{-1} \| e^{- \frac{|\omega|^2}{4\epsilon}} \|_{L^1} + 1 \big)  \| V_1 \|_{L^{1}(M_0)} 
\end{split}
\end{alignat}
can be bounded independently of $\epsilon$ in view of the Gaussian approximation of the identity, and 
\begin{alignat*}{2}
o_{L^{1},\epsilon}^{(\mu,\mu')} & (1) = \mathcal{F}^{-1}  e^{-\epsilon |\omega|^2}  \mathcal{F} [ \mathbf{1}_{D_r} a_{\mu}( \Phi_{|_{\tilde{\Omega}_{j'}}}^{-1}(z_{j'}) ) \bar{a}_{\mu'}( \Phi_{|_{\tilde{\Omega}_{j'}}}^{-1}(z_{j'}) ) V( \Phi_{|_{\tilde{\Omega}_{j'}}}^{-1}(z_{j'}) ) |g_{j'}|^{1/2} ] \\
 & - \mathbf{1}_{D_r}(z_0) [ a_{\mu}( \Phi_{|_{\tilde{\Omega}_{j'}}}^{-1}(z_{j'}) ) \bar{a}_{\mu'}( \Phi_{|_{\tilde{\Omega}_{j'}}}^{-1}(z_{j'}) ) V( \Phi_{|_{\tilde{\Omega}_{j'}}}^{-1}(z_{j'}) ) |g_{j'}|^{1/2}  ]_{|_{z_{j'} = z_0}} \rightarrow 0 \ \text{as $\epsilon \rightarrow 0$}
\end{alignat*}
in $L^{1}$ by the Fourier inversion theorem. Thus the limit as $\epsilon \rightarrow 0$ can be switched with the infinite sum in the second line of (\ref{estimate fourier recover 1}), from which we may conclude that
\begin{alignat}{2} \label{principle recovery impro}
\begin{split}
\lim_{\epsilon \rightarrow 0} \mathbf{1}_{\Omega_0}(p_0) \int_{\mathbb{C}} & e^{-\epsilon |\omega|^2}  e^{i \Phi(p_0) \cdot \omega}  \Big( \int_{\tilde{\Omega}_{j'}} e^{i\Phi \cdot \omega} |a|^2 V \hspace*{0.5mm} d\text{v}_{g} \Big) \hspace*{0.5mm} d\omega  d\bar{\omega} \\
& = \mathbf{1}_{D_{r}}(z_0) [  |a( \Phi_{|_{\tilde{\Omega}_{j'}}}^{-1} (z_{j'}) ; \Phi_{|_{{\Omega}_{0}}}^{-1} (z_0) )|^2 V(\Phi_{|_{\tilde{\Omega}_{j'}}}^{-1}(z_{j'})  |g_{j'}|^{1/2} ]_{_{z_{j'}=z_0}}
\end{split}
\end{alignat}
in $L^1(M)$. On the other hand, since
\begin{alignat*}{2}
\Phi \big( \Phi_{|_{\tilde{\Omega}_{j'}}}^{-1}(z_0) \big) = \Phi \big( \Phi_{|_{\Omega_{0}}}^{-1}(z_0)  \big) \ \ \text{for all} \ \ 1 \leq j' \leq N,
\end{alignat*}
by the construction of $a$ in (\ref{condition for a}), we obviously have the property
\begin{alignat}{2} \label{delta property of a}
a( \Phi_{|_{\tilde{\Omega}_{j'}}}^{-1}(z_{0}) ; \Phi^{-1}_{|_{\Omega_{0}}}(z_0) )  = 
\begin{cases}
1, \ \ \text{if} \ \ j' = 0, \ \ \text{and} \\[1mm]
0, \ \ \text{if} \ \ j' \neq 0,
\end{cases}
\end{alignat}
thus we have from (\ref{principle recovery impro}) that 
\begin{alignat}{2} \label{recovery princple cal improv part 2}
\lim_{\epsilon \rightarrow 0} \mathbf{1}_{\Omega_{0}}(p_0) \sum_{1 \leq j' \leq N} \int_{\mathbb{C}} e^{-\epsilon |\omega|^2} e^{i \Phi(p_0) \cdot \omega} \Big( \int_{\Omega_{j'}} e^{-i \Phi \cdot \omega} |a|^2 V \hspace*{0.5mm} d\text{v}_{g} \Big) \hspace*{0.5mm} d\omega  d\bar{\omega} = \mathbf{1}_{\Omega_0} V |g|^{1/2}
\end{alignat}
in $L^1(M)$. To complete the proof, we note that by construction, $\tilde{M}$ contains no point $p \in M_0$ for which $\Phi(p) \in D_{r}$ so that $\Phi(\tilde{M} \cap M_0)$ is disjoint from $D_r$. Thus with the same calculations as above we can conclude that 
\begin{alignat*}{2}
\lim_{\epsilon \rightarrow 0}   \mathbf{1}_{\Omega_0}(p_0)  &   \int_{\mathbb{C}}  e^{-\epsilon |\omega|^2}  e^{i \Phi(p_0) \cdot \omega} \Big( \int_{\tilde{M}} \chi_{j} e^{-i \Phi \cdot \omega} |a|^2 V \hspace*{0.5mm} d\text{v}_g \Big) \hspace*{0.5mm} d\omega d\bar{\omega}  = \sum_{\mu, \mu' \geq 0} \mathbf{1}_{D_{r}} \mathbf{1}_{\Phi( \tilde{M} \hspace*{0.5mm} \cap \hspace*{0.5mm} M_0 \hspace*{0.5mm} \cap \hspace*{0.5mm} \Omega_j  )} z_{0}^{\mu} \bar{z}_{0}^{\mu'} \\
& \hspace*{2cm} \times  [\chi_{j}(\Phi_{|_{\Omega_{j}}}^{-1}(z_{j}))  a_{\mu}( \Phi_{|_{\Omega_{j}}}^{-1}(z_{j}) ) \bar{a}_{\mu'}( \Phi_{|_{\Omega_{j}}}^{-1}(z_{j}) ) V(\Phi_{|_{\Omega_{j}}}^{-1} (z_{j})) |g_{j}|^{1/2}]_{|_{z_{j}=z_0}} = 0
\end{alignat*}
for each $j \geq 0$. Applying the above equation and (\ref{recovery princple cal improv part 2}) we arrive from (\ref{improv first expansion}) at (\ref{principle term calculation improve}).
\end{proof}

On the other hand, to take care of the last line in (\ref{first improv expansion}), we note that since $(1-\rho)$ is compactly supported, one can take the limit in $\epsilon$ directly inside the integral to arrive at 
\begin{alignat}{2} \label{smoother term in principle improv}
\begin{split}
\lim_{\epsilon \rightarrow 0} \mathbf{1}_{\Omega_0}(p_0) & \int_{|\omega| < |\omega_0|} (1-\rho) (\omega) e^{-\epsilon |\omega|^2} e^{-i \Phi(p_0) \cdot \omega} \Big( \int_{M_0} e^{-i \Phi \cdot \omega} |a|^2 V \hspace*{0.5mm} d\text{v}_g \Big) \hspace*{0.5mm} d\omega  d\bar{\omega} \\[1mm]
& = \mathbf{1}_{D_r}(z_0) \int_{|\omega| < |\omega_0|} (1-\rho)(\omega) e^{-iz_0 \cdot \omega} \Big( \int_{M_0} \chi_j e^{-i \Phi \cdot \omega} |a|^2 V d\text{v}_{g} \Big) \hspace*{0.5mm} d\omega d\bar{\omega}.
\end{split}
\end{alignat}
One is therefore able to differentiate under the integrals to conclude that the later integral depends smoothly on $z_0 \in D_r$. By compactness it follows that (\ref{smoother term in principle improv}) belongs to $L^{\infty}(\Omega_0)$. This concludes our analysis of the principle terms. 
\subsection{Analysis of Remainder Terms}
To take care of the right hand side of (\ref{modifiedi improv equation}), we refine the structure of the solutions constructed in Section 3. Let $\{ a_{\mu} \}_{\mu \geq 0}$ be the holomorphic coefficients of $a$. We set
\begin{alignat*}{2}
\tilde{u}_0^{(\mu)} \ \mydef \ a_{\mu}, \ \ \tilde{v}_{0}^{(\mu)} \ \mydef \ \bar{a}_{\mu} \ \ \text{and} \ \ \tilde{u}_{j}^{(\mu)} \ \mydef \ T_{\Phi} \tilde{\rho} T^{\star}(V_1 \tilde{u}_{j-1}^{(\mu)}), \ \ \tilde{v}_{j}^{(\mu)} \ \mydef \ \bar{T}_{\Phi} \tilde{\rho} \bar{T}^{\star}(V_2 \tilde{v}_{j-1}^{(\mu)}), \ \ j \geq 1
\end{alignat*}
and introduce the notations
\begin{alignat*}{2}
\tilde{I}_{k,k'} \ & \mydef \ \int_{M_0} e^{-i \Phi \cdot \omega} V \tilde{u}_{k} \tilde{v}_{k'} \hspace*{0.5mm} d\text{v}_{g} \ \ \text{and} \ \ \tilde{I}_{k,k'}^{\mu,\mu'} \ & \mydef \ \int_{M_0} e^{-i \Phi \cdot \omega} V \tilde{u}_{k}^{(\mu)} \tilde{v}_{k'}^{(\mu')}\hspace*{0.5mm} d\text{v}_{g}.
\end{alignat*}
Since $T_{\Phi} \tilde{\rho} T^{\star} : L^{p}(M_0) \rightarrow L^{\infty}(M_0)$ is bounded. Using (\ref{estimate for Taylor expansion}) it is clear that we have
\begin{alignat*}{2}
\tilde{u}_{j} =  \sum_{\mu \geq 0} \tilde{u}_{j}^{(\mu)} z_{0}^{\mu} \ \ \text{and} \ \ \tilde{v}_{j} =  \sum_{\mu \geq 0} \tilde{v}_{j}^{(\mu)} \bar{z}_{0}^{\mu}
\end{alignat*}
with convergence in $L^{\infty}(M_0)$ for all $p_0 \in \Omega_0$ and every $j \geq 0$. Our procedure for the cases $(k,k') = (1,0)$ or $(k,k') = (0,1)$ require additional arguments.
\begin{lemma} \label{remainder lemma improve}
The limits
\begin{alignat*}{2} 
\lim_{\epsilon \rightarrow 0} \mathbf{1}_{\Omega_0}(p_0) \int_{\mathbb{C}} e^{-\epsilon |\omega|^2 } e^{i \Phi(p_0) \cdot \omega} \tilde{I}_{1,0} \hspace*{0.5mm} d\omega d\bar{\omega} \ \ \text{and} \ \ \lim_{\epsilon \rightarrow 0} \mathbf{1}_{\Omega_0}(p_0) \int_{\mathbb{C}} e^{-\epsilon |\omega|^2} e^{i \Phi(p_0) \cdot \omega}  \tilde{I}_{0,1} \hspace*{0.5mm} d\omega d \bar{\omega}  
\end{alignat*}
exist in $L^{2}(M)$.
\end{lemma}
\begin{proof}
As before we extend $V_1, V_2$ to $M$ by zeros. Decomposing via the partition of unity $\{ \Omega_j, \chi_j, \chi_{j}' \}_{j \geq 0}$, since $\tilde{\rho}$ is compactly supported in $M_0'$, we have
\begin{alignat}{2}  \label{recovery princple cal improv part 3} 
\begin{split}
 \tilde{I}_{1,0} = \sum_{j \geq 0} &   \int_{M} \bar{a} V T ( \chi_j e^{-i \Phi \cdot \omega} \tilde{\rho}  T^{\star}  V_1 a) \hspace*{0.5mm} d\text{v}_{g} \\
& = \sum_{j \geq 0} \int_{M} \chi_{j}' \bar{a}V  R( \chi_j e^{-i \Phi \cdot \omega} \tilde{\rho} T^{\star} V_1 a)   \hspace*{0.5mm} d\text{v}_{g} + \sum_{j \geq 0}  \int_{M} \bar{a} V K_j ( \chi_j e^{-i \Phi \cdot \omega} \tilde{\rho} T^{\star} V_1 a ) \hspace*{0.5mm} d\text{v}_{g}.
\end{split}
\end{alignat}
Making the change of variables $\Phi_{|_{\Omega_j}}(p) = z_j$ and apply again the Taylor expansion of $a$, we can obtain from Fubini's theorem that
\begin{alignat}{2} \label{recovery princple cal improv part 4}
\begin{split}
\mathbf{1}_{\Omega_0} & (p_0) \int_{M} \chi_{j}' \bar{a}V R ( \chi_{j} e^{-i \Phi \cdot \omega} \tilde{\rho} T^{\star} V_1 a) \hspace*{0.5mm} d\text{v}_{g} + 1_{\Omega_0}(p_0) \int_{M} \bar{a}V K_j( \chi_j e^{-i \Phi \cdot \omega} \tilde{\rho}  T^{\star} V_1 a ) \hspace*{0.5mm} d\text{v}_{g}  \\[1mm]
& \hspace*{0.5cm} = - \mathbf{1}_{D_r}(z_0) \sum_{\mu, \mu' \geq 0} \bar{z}_0^{\mu} z_{0}^{\mu'} \int_{\Omega_j} \chi_j e^{-iz_j \cdot \omega} R( \chi_{j}' V \bar{a}_{\mu} |g_{j}|^{1/2} ) \tilde{\rho} T^{\star}(  V_1 a_{\mu'} ) \hspace*{0.5mm} dz_{j} d\bar{z}_{j}  \\
& \hspace*{3.2cm}  + \mathbf{1}_{D_r}(z_0) \sum_{\mu, \mu' \geq 0} \bar{z}_{0}^{\mu} z_{0}^{\mu'} \int_{\Omega_j} \chi_j e^{-i z_j \cdot \omega} K'_j( V \bar{a}_{\mu} ) \tilde{\rho} T^{\star} ( V_1 a_{\mu'} ) \hspace*{0.5mm} d\text{v}_g,
\end{split} 
\end{alignat}
for each $j \geq 0$, where we identify the $1$-forms $ \chi_j T^{\star}( V a_{\mu'} ) $ with their coefficients in the coordinates $z_j$ and $K_j'$ has smooth kernel. Multiplying on both sides of (\ref{recovery princple cal improv part 4}) by $e^{-\epsilon |\omega|^2} e^{i \Phi(p_0) \cdot \omega}$ and integrate over $\omega \in \mathbb{C}$, we have by the dominated convergence theorem that the second line of (\ref{recovery princple cal improv part 4}) becomes
\begin{alignat}{2} 
\begin{split} \label{recovery principle cal improv part 5}
- \hspace*{0.3mm} \mathbf{1}_{D_r}(z_0) \sum_{\mu, \mu' \geq 0} \bar{z}_0^{\mu} z_{0}^{\mu'} \int_{\mathbb{C}} e^{-\epsilon |\omega|^2 } e^{iz_0 \cdot \omega} \Big( \int_{\Omega_j} \chi_j e^{-iz_j \cdot \omega} R(\chi_j' V \bar{a}_{\mu} |g_{j}|^{1/2} ) \tilde{\rho} T^{\star} ( V_1 a_{\mu'} ) \hspace*{0.5mm} dz_j d\bar{z}_j \Big) \hspace*{0.5mm} d\omega d\bar{\omega}.
\end{split}
\end{alignat}
By the inequalities of H\"older and Sobolev, we can estimate
\begin{alignat}{2}
\begin{split} \label{this is L^2}
& \|  R(\chi_j' V \bar{a}_{\mu} |g_{j}|^{1/2} ) \chi_{j} \tilde{\rho} T^{\star}( V_1 a_{\mu'} ) \|_{L^2(\Omega_{j})} \\[1mm]
& \hspace*{1cm} \lesssim \| R(\chi_{j}' V \bar{a}_{\mu} |g_{j}|^{1/2} ) \|_{W^{1,4/3}(\Omega_j)} \| \chi_j T^{\star}( V_1 a_{\mu'} )  \|_{W^{1,4/3}(\Omega_j)} \\[1mm]
& \hspace*{1cm} \lesssim  \|\bar{a}_{\mu} \|_{L^{\infty}(M_0)} \| a_{\mu'} \|_{L^{\infty}(M_0)} \| V \|_{L^{4/3}(M_0)} \| V_1 \|_{L^{4/3}(M_0)}.
\end{split}
\end{alignat} 
Thus we may argue as in (\ref{estimate fourier recover 1}) via the Plancherel theorem and Fourier inversion to see that there exists a constant $C>0$ independent of $\epsilon >0$ and functions of order $o_{L^2,\epsilon}(1)$ such that
\begin{alignat*}{2}
\Big\| \int_{\mathbb{C}} & e^{-\epsilon |\omega|^2} e^{i z_0 \cdot \omega}  \Big( \int_{M} \chi_j'  \bar{a} V R( \chi_j e^{- i \Phi \cdot \omega} \tilde{\rho} T^{\star}V_1 a ) \hspace*{0.5mm} d\text{v}_{g} \Big) \hspace*{0.5mm} d\omega d\bar{\omega} \\
&  + \sum_{\mu, \mu' \geq 0} \bar{z}_{0}^{\mu} z_{0}^{\mu'}  [{\chi_{j} R( \chi_{j}' V \bar{a}_{\mu} |g_{j}|^{1/2} ) \tilde{\rho} T^{\star}( V_1 a_{\mu'} ) ]}_{|_{z_{j}=z_0}}  \Big\|_{L^{2}(D_r)}  \lesssim \sum_{\mu,\mu' \geq 0} r^{\mu + \mu'} \| o_{L^{2},\epsilon}(1) \|_{L^{2}(\mathbb{C})}  \\
& \hspace*{1cm} \lesssim \sum_{\mu, \mu' \geq 0} r^{\mu + \mu'} \sup_{\omega \in \mathbb{C}} | e^{-\epsilon |\omega|^2} - 1 | \| \mathcal{F}[ \chi_j R( \chi_j' V \bar{a}_{\mu} |g_j|^{1/2} ) \tilde{\rho} T^{\star}(V_1 a_{\mu'})  ] \|_{L^{2}(\mathbb{C})} \leq C < \infty,
\end{alignat*}
where $o_{L^2,\epsilon}(1) \rightarrow 0$ as $\epsilon \rightarrow 0$ in $L^2$. Summing over all $j \geq 0$ and taking the limit as $\epsilon \rightarrow 0$ in the above, we have form (\ref{recovery princple cal improv part 4}) that
\begin{alignat}{2}
\begin{split} \label{yet another convergence}
\lim_{\epsilon \rightarrow 0 } \sum_{j \geq 0 } \int_{\mathbb{C}} & e^{-\epsilon |\omega|^2} e^{i z_0 \cdot \omega}  \Big( \int_{M}  \chi_j'  \bar{a} V R( \chi_j e^{- i \Phi \cdot \omega} \tilde{\rho} T^{\star}V_1 a ) \hspace*{0.5mm} d\text{v}_{g} \Big) \hspace*{0.5mm} d\omega d\bar{\omega} \\
 & \hspace*{0.5cm} = - \mathbf{1}_{D_{r}}(z_0) \sum_{j \geq 0} \sum_{\mu, \mu' \geq 0} \bar{z}_0^{\mu} z_0^{\mu'} [{\chi_j R(\chi_j' V \bar{a}_{\mu} |g_{j}|^{1/2} ) \tilde{\rho} T^{\star}( V_1 a_{\mu'} )}]_{|_{z_{j}=z_0}}
\end{split}
\end{alignat}
in $L^2(M)$. The right hand side of (\ref{yet another convergence}) lives in $L^2(D_r)$ by estimate (\ref{this is L^2}) and (\ref{estimate for Taylor expansion}). Thus we have arrived at the required limit. The last line in (\ref{recovery princple cal improv part 4}) involving smoothing operators can be taken cared of in the same way, and the obvious modification to the argument works for the term which contains $\tilde{I}_{0,1}$.
\end{proof}

Finally we can again write 
\begin{alignat}{2} \label{recovery remainder cal improv part 1}
\begin{split}
& \mathbf{1}_{\Omega_0}(p_0) \int_{|\omega| \geq |\omega_0|}  \rho(\omega) e^{-\epsilon |\omega|^2 } e^{i \Phi(p_0) \cdot \omega} \tilde{I}_{1,0} \hspace*{0.5mm} d\omega d\bar{\omega} \\
&  =   \mathbf{1}_{\Omega_0}(p_0) \int_{\mathbb{C}} e^{-\epsilon |\omega|^2} e^{i \Phi(p_0) \cdot \omega} \tilde{I}_{1,0} \hspace*{0.5mm} d\omega  d\bar{\omega} -  \mathbf{1}_{\Omega_0}(p_0) \int_{|\omega| < |\omega_0|} (1-\rho)(\omega) e^{-\epsilon |\omega|^2} e^{i \Phi(p_0) \cdot \omega} \tilde{I}_{1,0} \hspace*{0.5mm} d\omega  d\bar{\omega}.
\end{split}
\end{alignat}
Notice from (\ref{global splitting}) that $\tilde{I}_{1,0}$ depends smoothly on $p_0 \in \Omega_0$. Thus we can apply Lemma \ref{remainder lemma improve} to the first term above and argue as in (\ref{smoother term in principle improv}) for the second shows that (\ref{recovery remainder cal improv part 1}) converges to a limit in $L^2(M)$. A similar calculation works for the case of $\tilde{I}_{0,1}$. \par 
It remains to show that the limit
\begin{alignat}{2} \label{final convergence}
\lim_{\epsilon \rightarrow 0} \mathbf{1}_{\Omega_0}(p_0) \sum_{k+k' \geq 2} \int_{|\omega| \geq |\omega_0|} \rho(\omega) e^{-\epsilon |\omega|^2} e^{i \Phi(p_0) \cdot \omega} \tilde{I}_{k,k'} \hspace*{0.5mm} d\omega d\bar{\omega}
\end{alignat}
exists in $L^{2}(M_0)$. For this we will make use of the Carleman estimates proved in Section 3. For the cases where $k=0$ or $k'=0$ we formulate the following result.
\begin{lemma} \label{lower order terms improv lemma}
There exists constants $\tilde{C}, C>0$ independent of $\lambda$ and $p_0$ such that
\begin{alignat}{2} 
\begin{split} \label{lower order term improv cal}
& |\tilde{I}_{k,0}^{\mu,\mu'}|  \leq \tilde{C}  \Big( \frac{C \| V_1 \|_{L^{p}(M_0)}}{|\omega|^{0+}} \Big)^{k-2} \frac{\| a_{\mu} \|_{L^{\infty}(M_0)} \| \bar{a}_{\mu'} \|_{L^{\infty}(M_0)} }{|\omega|^{1+}} \ \ \text{and} \\
 & \hspace*{0.42cm} |\tilde{I}_{0,k'}^{\mu,\mu'}|  \leq \tilde{C} \Big( \frac{C \| V_2 \|_{L^{p}(M_0)}}{|\omega|^{0+}} \Big)^{k'-2} \frac{ \| a_{\mu} \|_{L^{\infty}(M_0)} \| \bar{a}_{\mu'} \|_{L^{\infty}(M_0)} }{|\omega|^{1+}}
\end{split}
\end{alignat}
for all $k, k' \geq 2$.
\end{lemma}
\begin{proof}
It is sufficiently to note from Corollary \ref{Carleman estimate corollary} and H\"older's inequalities that
\begin{alignat*}{2}
\int_{M_0} |V| & |\tilde{u}_{k}^{(\mu)}| |a_{\mu'}| \hspace*{0.5mm} d\text{v}_{g} \lesssim  \| \bar{a}_{\mu'} V \|_{L^p} \| T_{\Phi} \tilde{\rho} T^{\star}( V_1 T_{\Phi} \tilde{\rho} T^{\star} \tilde{u}_{k-2}^{(\mu)} V_1 ) \|_{L^{q}} \\
& \lesssim \frac{\| \bar{a}_{\mu'} \|_{L^{\infty}} \| V \|_{L^p} \| V_1 \|_{L^p} \| T_{\Phi} \tilde{\rho} T^{\star} \tilde{u}_{k-2}^{(\mu)} V_1  \|_{L^{\infty}} }{|\omega|} \lesssim \frac{\| \bar{a}_{\mu'} \|_{L^{\infty}} \| V \|_{L^{p}} \| V_1 \|_{L^{p}}^2 \| \tilde{u}_{k-2}^{(\mu)} \|_{L^{\infty}} }{|\omega|^{1+}}
\end{alignat*}
where $q$ is the H\"older conjugate of $p$. Now by iterating the $L^{p} \rightarrow L^{\infty}$ estimate obtained in (\ref{Carleman estimate for the Laplacian}), we easily see that there exists a constant $C>0$ such that
\begin{alignat}{2} \label{iterative estimate}
\| \tilde{u}_{k}^{(\mu)} \|_{L^{\infty}} \leq \Big( \frac{C \| V_1 \|_{L^{p}}}{|\omega|^{0+}} \Big)^{k} \| a_{\mu} \|_{L^{\infty}}
\end{alignat}
for all $k \geq 0$. Putting the two estimates together we arrive at the required claim. A similar calculation works for the case of $\tilde{I}_{0,k'}^{\mu,\mu'}$.
\end{proof}

The remaining terms in (\ref{final convergence}) can be estimated in similar manners.
\begin{lemma} \label{high order terms improve cal}
There exists constant $\tilde{C}, C>0$ independent of $\lambda$ and $p_0$ such that
\begin{alignat}{2}
\begin{split}
|\tilde{I}_{k,k'}^{\mu,\mu'}|  \leq \tilde{C} \Big( \frac{C \max\{ \| V_1 \|_{L^p(M_0)}, \| V_2 \|_{L^{p}(M_0)} \}}{|\omega|^{0+}} \Big)^{k+k'-2} \frac{\| a_{\mu} \|_{L^{\infty}(M_0)} \| \bar{a}_{\mu'} \|_{L^{\infty}(M_0)}}{|\omega|^{1+}}
\end{split}
\end{alignat}
for all $k+k' \geq 2$ with $k, k' \neq 0$.
\end{lemma}
\begin{proof}
Applying again Corollary \ref{lower order terms improv lemma}, we see that if $1/p + 1/q = 1$, then
\begin{alignat*}{2}
\int_{M_0} |V| & |\tilde{u}_{k}^{(\mu)}| |\tilde{v}_{k'}^{(\mu')}| \hspace*{0.5mm} d\text{v}_{g} \leq \| V \|_{L^{p}} \| T_{\Phi} \tilde{\rho} T^{\star}(\tilde{u}_{k-1}^{(\mu)} V_1) \|_{L^{q}} \| \bar{T}_{\Phi} \tilde{\rho} \bar{T}^{\star}( \tilde{v}_{k'-1}^{(\mu')} V_2 ) \|_{L^{\infty}} \\
& \lesssim \frac{\| V \|_{L^{p}} \| \tilde{u}_{k-1}^{(\mu)} V_1 \|_{L^{p}} \| \tilde{v}_{k'-1}^{(\mu')} V_2 \|_{L^{p}} }{|\omega|^{1+}} \lesssim \| \tilde{u}_{k-1}^{(\mu)} \|_{L^{\infty}} \| \tilde{v}_{k'-1}^{(\mu')} \|_{L^{\infty}} \frac{\| V \|_{L^{p}} \| V_1 \|_{L^{p}} \| V_2 \|_{L^{p}} }{|\omega|^{1+}}.
\end{alignat*}
Now (\ref{iterative estimate}) yields the existence of a constant $C>0$ such that
\begin{alignat*}{2}
\| \tilde{u}_{k-1}^{(\mu)} \|_{L^{\infty}} \| \tilde{v}_{k'-1}^{(\mu')} \|_{L^{\infty}} \lesssim \Big( \frac{ C \max\{ \| V_1 \|_{L^{p}}, \| V_2 \|_{L^{p}} \} }{|\omega|^{0+}} \Big)^{k+k' - 2} \| a_{\mu} \|_{L^{\infty}} \| \bar{a}_{\mu'} \|_{L^{\infty}}
\end{alignat*}
which implies the claim.
\end{proof}

Combining the results of Lemma \ref{lower order terms improv lemma} and Lemma \ref{high order terms improve cal}, we see now that there exists constants $\tilde{C}, C>0$ such that whenever $k+k' \geq 2$, we have
\begin{alignat}{2} 
\begin{split} \label{improv total cal}
|\tilde{I}_{k,k'}^{\mu,\mu'}|  \leq \tilde{C} \Big( \frac{C \max\{ \| V_1 \|_{L^{p}(M_0)}, \| V_2 \|_{L^{p}(M_0)} \} }{|\omega|^{0+}} \Big)^{k+k' - 2} \frac{\| a_{\mu} \|_{L^{\infty}(M_0)} \| a_{\mu'} \|_{L^{\infty}(M_0)}}{|\omega|^{1+}}.
\end{split}
\end{alignat} 
Moreover, from (\ref{improv total cal}) we deduce that $\rho \tilde{I}_{k,k'}^{\mu,\mu'} \in L^2$ in the $\omega$ variable, so that by the Plancherel theorem we know $\mathcal{F}^{-1} \rho \tilde{I}_{k,k'}^{\mu,\mu'}$ exists in $L^2$. Since we have
\begin{alignat*}{2}
 \mathbf{1}_{\Omega_0}(p_0) \int_{|\omega| \geq |\omega_0|} & e^{-\epsilon |\omega|^2} e^{i \Phi(p_0) \cdot \omega} \rho(\omega) \tilde{I}_{k,k'} \hspace*{0.5mm} d\omega d\bar{\omega} \\
& \hspace*{0.58cm} =  \sum_{\mu, \mu' \geq 0} \mathbf{1}_{D_r}(z_0) z_{0}^{\mu} \bar{z}_{0}^{\mu'} \int_{|\omega| \geq |\omega_0 |} e^{-\epsilon |\omega|^2} e^{i z_0 \cdot \omega} \rho(\omega) \tilde{I}_{k,k'}^{\mu,\mu'} \hspace*{0.5mm} d\omega d\bar{\omega},
\end{alignat*}
it follows from (\ref{estimate for Taylor expansion}) that there exists $C'>0$ such that
\begin{alignat*}{2}
\Big\|  \int_{|\omega| \geq |\omega_0|} &  e^{-\epsilon |\omega|^2}  e^{i z_0 \cdot \omega}  \rho(\omega) \sum_{k+k' \geq 2} \tilde{I}_{k,k'} \hspace*{0.5mm} d\omega d\bar{\omega} - \sum_{k+k' \geq 2} \sum_{\mu, \mu' \geq 0} z_0^{\mu} \bar{z}_{0}^{\mu'} \mathcal{F}^{-1}  \rho \tilde{I}_{k,k'}^{\mu,\mu'} \hspace*{0.5mm} \Big\|_{L^2(D_r)} \\[1mm]
& \lesssim \sum_{k+k' \geq 2} \sum_{\mu, \mu' \geq 0} r^{\mu+\mu'} \| \mathcal{F}^{-1} \big( e^{-\epsilon |\omega|^2} - 1\big) \rho \tilde{I}_{k,k'}^{\mu,\mu'}  \|_{L^2(D_r)} \\
& \lesssim \sup_{\omega \in \mathbb{C}} \frac{ | e^{-\epsilon |\omega|^2} - 1 |}{|\omega|^{1+}} \sum_{k+k' \geq 2} \Big( \frac{C'}{|\omega|^{0+}}  \Big)^{k+k'-2} \sum_{\mu, \mu' \geq 0} r^{\mu+\mu'} \| a_{\mu} \|_{L^{\infty}(M_0)} \| a_{\mu'} \|_{L^{\infty}(M_0)}    < \infty
\end{alignat*}
for all $|\omega| > |\omega_0|$. In particular, the last line above is uniformly bounded in $\epsilon > 0$, so Fourier inversion yields
\begin{alignat*}{2}
\lim_{\epsilon \rightarrow 0} \mathbf{1}_{\Omega_0}(p_0) \int_{|\omega| \geq |\omega_0|} & e^{-\epsilon |\omega|^2} e^{i \Phi(p_0) \cdot \omega} \rho(\omega) \sum_{k+k' \geq 2} \tilde{I}_{k,k'} \hspace*{0.5mm} d\omega d\bar{\omega} = \sum_{k+k' \geq 2} \sum_{\mu, \mu' \geq 0} \mathbf{1}_{D_r}(z_0) z_0^{\mu} \bar{z}_{0}^{\mu'} \mathcal{F}^{-1}\rho \tilde{I}_{k,k'}^{\mu,\mu'}
\end{alignat*}
in $L^{2}(M)$ as expected. From (\ref{improv total cal}) we also have
\begin{alignat*}{2}
\Big\| & \sum_{k+k' \geq 2} \sum_{\mu, \mu' \geq 0} |z_0|^{\mu + \mu'} \mathcal{F}^{-1} \rho \tilde{I}_{k,k'}^{\mu,\mu'} \Big\|_{L^{2}(D_r)} \leq \sum_{k+k' \geq 2} \sum_{\mu, \mu' \geq 0} r^{\mu+\mu'} \| \rho \tilde{I}_{k,k'^{\mu,\mu'}} \|_{L^2} \\
& \leq \Big( \sum_{k+k' \geq 2} 2^{-k-k'+2} \Big) \sum_{\mu,\mu' \geq 0} \tilde{C} r^{\mu+\mu'} \| a_{\mu} \|_{L^{\infty}(M_0)} \| a_{\mu'} \|_{L^{\infty}(M_0)} \Big( \int_{|\omega| > |\omega_0|} \frac{|\rho(\omega)|^2}{|\omega|^{2+}} \hspace*{0.5mm} |d\omega d\bar{\omega}| \Big)^{1/2} 
\end{alignat*}
which is finite for sufficiently large $|\omega_0|$, hence we have arrived at the required convergence in (\ref{final convergence}). Putting everything together and noting that $L^{2}(M) \hookrightarrow L^{1}(M)$, we can take the $L^1(M)$ limit on both sides of (\ref{Improvement expansion}) to deduce that
\begin{alignat*}{2}
\mathbf{1}_{\Omega_0} V |g|^{1/2} \in L^{2}(M).
\end{alignat*}
Since $|g|^{1/2}$ is non-vanishing, this in particular implies that $V$ is $L^2$ on $\Omega_0$ and so we have arrived at the required claim.
\end{proof}

\section{Identification of the Potential}
In this final section we prove Theorem \ref{main theorem}. The procedure will be similar to what was done in Section 4, with the key difference being that we now have $V \in L^{2}(M_0)$. As in section 4 it suffices to show
\begin{proposition}
Any point in $M_0$ admits an open neighbourhood $\Omega_0 \subset M$ such that $V = 0$ almost everywhere on $\Omega_0$.
\end{proposition}
\begin{proof}
For an arbitrary point $\tilde{p}_0$ in $M_0$ we adopt the same convention introduced at the beginning of Section 4. We recall that this means we let $ \{ \tilde{\Omega}_{j'}, \tilde{\chi}_{j'}, \tilde{\chi}_{j'}' \}_{ 1 \leq j' \leq N} $ be constructed as in Subsection 2.3. The collection $\{ \Omega_j \}_{j \geq 0}$ defines an open covering of $M_0'$ in $M$ so that for each $j \geq 0$, the map $\Phi: \Omega_{j} \rightarrow \Phi(\Omega_{j})$ is biholomorphic. We let $\{ \chi_{j} \}_{j \geq 0}$ be a partition of unity subordinate to $\{ \Omega_{j} \}_{j \geq 0}$, and choose $ \{ \chi_{j}' \}_{j \geq 0}$ so for that each $j \geq 0$, $\chi_{j}'$ supported on a holomorphic coordinate neighbourhood of $\Omega_j$ with coordinate map $\Phi$ and is identically $1$ on the support of $\chi_j$. Lastly, $\tilde{M}$ is the complement of the union of $\{ \tilde{\Omega}_{j'} \}_{1 \leq j' \leq N}$ in $M$. 
\par By implementing solutions $u$ and $v$ from Proposition \ref{existence of CGO solutions} into (\ref{Alessndrini identity}), we have 
\begin{alignat}{2} \label{Main Theorem Expansion}
0 = \int_{M_0} u V v \hspace*{0.5mm} d\text{v}_{g} = \int_{M_{0}} e^{2i \psi \lambda} |a|^2 V \hspace*{0.5mm} d\text{v}_{g} + \sum_{k + k' \geq 1}  \int_{M_0} e^{2i \psi \lambda} u_{k} V v_{k'} \hspace*{0.5mm} d\text{v}_{g} .
\end{alignat}
Multiplying both sides of (\ref{Main Theorem Expansion}) by $2 \lambda \pi^{-1} \mathbf{1}_{\Omega_0}(p_0)$ and rearranging gives 
\begin{alignat}{2} \label{weighted main theorem expansion}
\frac{2 \lambda \mathbf{1}_{\Omega_0}(p_0)}{\pi} \int_{M_0} e^{2i \psi \lambda} |a|^2 V \hspace*{0.5mm} d\text{v}_{g} = - \frac{2 \lambda \mathbf{1}_{\Omega_0}(p_0)}{\pi}  \sum_{k + k' \geq  1} \int_{M_0} e^{2i \psi \lambda} u_{k} V v_{k'} \hspace*{0.5mm} d\text{v}_{g}.
\end{alignat}
In order to identity $V$ from this expression, we will exploit the following $L^{2}$-based method of stationary phase, the proof of which we recall from Lemma 3.3 in \cite{L2calderon}.
\begin{lemma}
Let $V$ be in $L^2(\mathbb{C})$ and $(V_{\delta})_{\delta > 0}$ be a smooth approximation of $V$ in $L^2(\mathbb{C})$, then for any $s \in [0,1]$ we have
\begin{alignat}{2} \label{stationary phase epislon estimate}
\| V - \frac{2\lambda}{\pi} e^{2i \text{Re} \hspace*{0.5mm} z^2 \lambda} \star V \|_{L^2} \lesssim \| V - V_{\delta} \|_{L^2} + \frac{\| V_{\delta} \|_{H^{s}}}{\lambda^{s/2}}.
\end{alignat}
In particular, we have 
\begin{alignat}{2} \label{stationary phase equation}
\lim_{\lambda \rightarrow \infty} \frac{2 \lambda}{\pi}  e^{2i \text{Re} \hspace*{0.5mm} z^2 \lambda} \star V = V
\end{alignat}
in $L^{2}(\mathbb{C})$.
\end{lemma}
\begin{proof}
By a standard calculation of the complex expoential and convolution theorem, we have
\begin{alignat*}{2}
\mathcal{F} \Big( \frac{2 \lambda}{\pi} e^{ \pm 2i \text{Re} \hspace*{0.5mm} z^2 \lambda} \star V \Big) = \mathcal{F}e^{\pm 2i \text{Re} \hspace*{0.5mm} z^2 \lambda} \mathcal{F}V =  e^{\mp \frac{i \text{Re} \hspace*{0.5mm} \zeta^{2}}{8 \lambda}} \mathcal{F}V.
\end{alignat*}
Using the Fourier-Plancherel Theorem, we have
\begin{alignat}{2} \label{stationary phase calculation 1}
\begin{split}
\| \hspace*{0.2mm} V - \frac{2 \lambda }{\pi} e^{\pm 2i \text{Re} \hspace*{0.5mm} z^{2} \lambda} & \star V \hspace*{0.2mm} \|_{L^{2}} = \|( 1- e^{\mp \frac{i \text{Re} \hspace*{0.5mm} \zeta^2}{8 \lambda}} ) \mathcal{F}V \|_{L^{2}} .
\end{split}
\end{alignat}
Now for all $s \in [0,1]$ we may estimate that $|1 -  e^{\mp 2i \text{Re} \hspace*{0.5mm} \zeta^2  }| \lesssim 2^{s/2} | \zeta |^s$. Indeed, if $|\zeta| \geq 1$, then it is easy to see that $|1- e^{\mp 2i \text{Re} \hspace*{0.5mm} \zeta^2}| \leq 2 \lesssim 2^{s/2} |\zeta|^{s}$ and so the result is obvious. On the other hand, if $|\zeta| \leq 1$, then a direct computation yields
\begin{alignat*}{2}
|1 - e^{\mp 2i \text{Re} \hspace*{0.5mm} \zeta^2 } |^2 = 4 |\sin( \xi^2 - \eta^2 )|^2,
\end{alignat*}
therefore we have that 
\begin{alignat*}{2}
|1- e^{\mp 2i \text{Re} \hspace*{0.5mm} \zeta^2}|^2 \lesssim | \xi^2 - \eta^2 |^2 \leq | \xi^2 + \eta^2 |^2 \leq 2^{s} |\zeta|^{2s},
\end{alignat*}
By combining the above inequalities, we can extend (\ref{stationary phase calculation 1}) to 
\begin{alignat}{2} \label{stationary phase calcuation 2}
\big\| \hspace*{0.2mm}  V - \frac{2\lambda}{\pi}   e^{2i \text{Re} z^2 \lambda} \star  V  \big\|_{L^{2}}  \lesssim  \frac{\|  (1+ |\zeta|^2)^{s} \mathcal{F} V \|_{L^2}}{\lambda^{s/2}} = \frac{\| V \|_{H^{s}}}{\lambda^{s/2}}.
\end{alignat}
Now let $\{ V_{\delta} \}_{\delta > 0}$ be a sequence of smooth functions in $H^{s}$ such that $\| V - V_{\delta} \|_{H^{s}} < \delta$. By inequality (\ref{stationary phase calcuation 2}) we now have
\begin{alignat*}{2}
\| (V- V_{\delta}) - \frac{2 \lambda }{\pi} e^{2i \text{Re} \hspace*{0.5mm} z^2 \lambda} \star (V- V_{\delta})  \|_{L^{2}} \lesssim \| V - V_{\delta} \|_{L^{2}} < \delta. 
\end{alignat*}
Triangle inequality now gives
\begin{alignat}{2} \label{stationary phase calculation 3}
\| V - & \frac{2 \lambda}{\pi} e^{2i \text{Re} \hspace*{0.5mm} z^{2} \lambda} \star V \|_{L^{2}}  \lesssim \| V_{\delta} - V \|_{L^{2}} + \| V_{\delta} - \frac{2\lambda}{\pi} e^{2i \text{Re} \hspace*{0.5mm} z^{2} \lambda} \star V_{\delta} \|_{L^{2}} < \delta + \frac{\| V_{\delta} \|_{H^{s}}}{\lambda^{s/2}}.
\end{alignat}
This proves (\ref{stationary phase epislon estimate}). Letting $\lambda \rightarrow \infty$ followed by $\delta \rightarrow 0$ and noticing the left hand side of (\ref{stationary phase calculation 3}) is independent of $\delta$ concludes the proof of (\ref{stationary phase equation}).
\end{proof}

Our strategy now follows in a similar way as Section 4. Using Lemma \ref{stationary phase equation}, the left hand side of (\ref{weighted main theorem expansion}) will converge to $\mathbf{1}_{\Omega_0} V |g|^{1/2}$ in $L^2(M)$ as $\lambda \rightarrow \infty$ while the right hand side vanishes in the limit. The same problem regarding the dependency of $a$ on both $p$ and $p_0$ remains, and we get over this issue with again the Taylor expansion of $a$ derived from Lemma \ref{Taylor expansion Lemma}. 
\subsection{Analysis of Principle Terms}
In this subsection we study the integral
\begin{alignat}{2} \label{principle integral} 
\frac{2 \lambda \mathbf{1}_{\Omega_0}(p_0)}{\pi} \int_{M_0} e^{2i \psi \lambda} |a|^2 V \hspace*{0.5mm} d\text{v}_{g}
\end{alignat}
from which we will recover the information of $V$ on $\Omega_0$. \par  Extending $V_1, V_2$ to $M$ by zero, we prove the following result analogous to Lemma \ref{principle term calculation improve lemma}. 
\begin{lemma} \label{stationary phase asym}
We have that 
\begin{alignat*}{2}
\lim_{\lambda \rightarrow 0}  \frac{2 \lambda \mathbf{1}_{\Omega_0}(p_0)}{\pi} \int_{M} e^{2i \psi \lambda} |a|^2 V \hspace*{0.5mm} d\text{\emph v}_{g} = \mathbf{1}_{\Omega_{0}} V |g|^{1/2}
\end{alignat*}
in $L^{2}(M)$.
\end{lemma}
\begin{proof}
We may write (\ref{principle integral}) as 
\begin{alignat}{2} \label{principle integral cal 1}
\begin{split}
& \frac{2 \lambda \mathbf{1}_{\Omega_0}(p_0)}{\pi} \int_{M} e^{2i \psi \lambda} |a|^2 V \hspace*{0.5mm} d\text{v}_{g} \\
& \hspace*{1cm} = \sum_{j' \geq 0} \frac{2 \lambda \mathbf{1}_{\Omega_0}(p_0)}{\pi} \int_{\tilde{\Omega}_{j' }} e^{2i \psi \lambda} |a|^2 V \hspace*{0.5mm} d\text{v}_{g} + \sum_{j \geq 0} \frac{2 \lambda \mathbf{1}_{\Omega_0}(p_0)}{\pi} \int_{\tilde{M}} \chi_{j} e^{2i \psi \lambda} |a|^2 V \hspace*{0.5mm} d\text{v}_{g}.
\end{split}
\end{alignat}
By making the change of variables $\Phi_{|_{\tilde{\Omega}_{j'}}}(p) = z_{j'}$, we can apply Lemma \ref{Taylor expansion Lemma} and the dominated convergence theorem to see that for every $1 \leq j' \leq N$, we have
\begin{alignat}{2} \label{principle integral cal 2}
\begin{split}
& \frac{2 \lambda \mathbf{1}_{\Omega_0}(p_0)}{\pi} \int_{\tilde{\Omega}_{j'}} e^{2i \psi \lambda} |a|^2 V \hspace*{0.5mm} d\text{v}_{g} \\
& \hspace*{1cm} = \frac{2 \lambda \mathbf{1}_{D_{r}}(z_0)}{\pi} \int_{D_r} e^{2i \text{Re} \hspace*{0.5mm} (z_{j'}-z_0)^2 \lambda} |a( \Phi_{|_{\tilde{\Omega}_{j'}}}^{-1}(z_{j'}); \Phi_{|_{\Omega_{0}}}^{-1}(z_{0}) )|^2 V( \Phi_{|_{\tilde{\Omega}_{j'}}}^{-1}(z_{j'}) ) \hspace*{0.5mm} |g_{j'}|^{1/2} dz_{j'} d\bar{z}_{j'} \\[1mm]
& \hspace*{2cm} = \sum_{\mu, \mu' \geq 0} \frac{ 2 \lambda \mathbf{1}_{D_r}(z_0) }{\pi} z_0^{\mu} \bar{z}_{0}^{\mu'} \int_{D_r} e^{2i \text{Re} \hspace*{0.5mm} (z_{j'}-z_0)^{2} \lambda} \\
& \hspace*{5.2cm} \times a_{\mu}( \Phi_{|_{\tilde{\Omega}_{j'}}}^{-1}(z_{j'}) )  \bar{a}_{\mu'}( \Phi_{|_{\tilde{\Omega}_{j'}}}^{-1}(z_{j'}) ) V( \Phi_{|_{\tilde{\Omega}_{j'}}}^{-1}(z_{j'}) ) \hspace*{0.5mm} |g_{j'}|^{1/2} dz_{j'} d\bar{z}_{j'}. 
\end{split}
\end{alignat}
We want to take a $L^{2}(M)$ limit as $\lambda \rightarrow \infty$ in the last sum of (\ref{principle integral cal 2}). For this we note that since $V$ is in $L^{2}(M_0)$, we have from inequality (\ref{solution estimates}), (\ref{stationary phase epislon estimate}) and (\ref{stationary phase equation}) the estimate
\begin{alignat*}{2} 
&\Big\| \frac{2\lambda}{\pi} \int_{D_r}  e^{2i \text{Re} \hspace*{0.5mm} (z_{j'} - z_0)^2 \lambda} |a( \Phi_{|_{\tilde{\Omega}_{j'}}}^{-1}(z_{j'}); \Phi_{|_{\Omega_{0}}}^{-1}(z_0) )|^2 V( \Phi_{|_{\tilde{\Omega}_{j'}}}^{-1}(z_{j'}) ) \hspace*{0.5mm} |g_{j'}|^{1/2} dz_{j'} d\bar{z}_{j'} \\[1mm]
& \hspace*{2.6cm} -\sum_{\mu, \mu' \geq 0}  z_{0}^{\mu} \bar{z}_{0}^{\mu'} [a_{\mu}( \Phi_{|_{\tilde{\Omega}_{j'}}}^{-1}(z_{j'}) ) \bar{a}_{\mu'}( \Phi_{|_{\tilde{\Omega}_{j'}}}^{-1}(z_{j'}) ) V( \Phi_{|_{\tilde{\Omega}_{j'}}}^{-1}(z_{j'}) ) |g_{j'}|^{1/2}]_{|_{z_{j'} = z_0}} \hspace*{0.5mm} \Big\|_{L^{2}(D_r)} \\
& \leq \sum_{\mu,\mu' \geq 0} r^{\mu + \mu'} \| o_{L^2,\lambda}^{(\mu,\mu')}(1) \|_{L^2}  \lesssim \sum_{\mu,\mu' \geq 0} r^{\mu+\mu'} \| a_{\mu} \|_{L^{\infty}(M_0)} \| a_{\mu'} \|_{L^{\infty}(M_0)} \| V \|_{L^2} \lesssim \| V \|_{L^2(M_0)} < \infty,
\end{alignat*}
where $o_{L^2, \lambda}^{(\mu,\mu')} \rightarrow 0$ as $\lambda \rightarrow \infty$ in $L^{2}$ depending on $\mu,\mu' \geq 0$. It follows from (\ref{principle integral cal 2}) that the above calculation implies
\begin{alignat}{2}
\begin{split} \label{identification cal}
\lim_{\lambda \rightarrow \infty} &\frac{2 \lambda  \mathbf{1}_{\Omega_0}(p_0)}{\pi}  \int_{\tilde{\Omega}_{j'}} e^{2i \psi \lambda} |a|^2 V \hspace*{0.5mm} d\text{v}_{g} \\
& \hspace*{0.5cm} = 1_{D_r}(z_0) [|a( \Phi_{|_{\tilde{\Omega}_{j'}}}^{-1}(z_{j'}) ; \Phi_{|_{\Omega_0}}^{-1}(z_0) )|^2 V( \Phi_{|_{\tilde{\Omega}_{j'}}}^{-1}(z_{j'}) ) |g_{j'}|^{1/2}]_{z_{j'} = z_0}
\end{split}
\end{alignat}
in $L^2(M)$ for every $1 \leq j' \leq N$. Thanks to (\ref{delta property of a}), summing over all such $j'$ in (\ref{identification cal}) yields
\begin{alignat}{2} \label{correct asymp}
\lim_{\lambda \rightarrow \infty} \frac{2 \lambda \mathbf{1}_{\Omega_0}(p_0)}{\pi} \sum_{j' \geq 0} \int_{\tilde{\Omega}_{j'}} e^{2i \psi \lambda}|a|^2 V \hspace*{0.5mm} d\text{v}_{g} = \mathbf{1}_{\Omega_0} V |g|^{1/2}
\end{alignat}
which is the required asymptotic. Now since $\tilde{M}$ contains no point $p \in M_0$ for which we have $\Phi(p) \in D_{r}$, we must have $\Phi(\tilde{M}  \cap  M_0 )$ and $D_r$ are disjoint. Thus it is easy to see from the same arguments as above that
\begin{alignat}{2}
\begin{split} \label{principle integral cal 4}
\lim_{\lambda \rightarrow \infty} \frac{2 \lambda  \mathbf{1}_{\Omega_0}(p_0)}{\pi} & \int_{\tilde{M}} \chi_{j} e^{2i \psi \lambda} |a|^2 V \hspace*{0.5mm} d\text{v}_{g} = \sum_{\mu, \mu' \geq 0} \mathbf{1}_{D_r} \mathbf{1}_{\Phi ( \tilde{M} \hspace*{0.5mm} \cap \hspace*{0.5mm} M_0 \hspace*{0.5mm} \cap \hspace*{0.5mm} \Omega_{j} ) } z_0^{\mu} \bar{z}_{0}^{\mu'} \\
& \times [\chi_{j}( \Phi_{|_{\Omega_j}}^{-1}(z_j) )a_{\mu}( \Phi_{|_{\Omega_{j}}}^{-1}(z_{j}) ) \bar{a}_{\mu'}(\Phi_{|_{\Omega_{j}}}^{-1}(z_{j})) V(\Phi_{|_{\Omega_{j}}}^{-1}(z_{j})) |g_{j}|^{1/2}]_{z_{j}=z_0} = 0
\end{split}
\end{alignat}
for each $j \geq 1$. Combining (\ref{principle integral cal 1}) with (\ref{correct asymp}) and (\ref{principle integral cal 4}) concludes the proof of the claim.
\end{proof}

\subsection{Analysis of Remainder Terms} In this subsection we show that
\begin{alignat*}{2}
\lim_{\lambda \rightarrow \infty} \frac{2 \lambda \mathbf{1}_{\Omega_0}(p_0)}{\pi} \sum_{k+k' \geq 1} I_{k,k'} = o_{L^2, \epsilon}(1)
\end{alignat*}
where $o_{L^2, \epsilon}(1) \rightarrow 0$ as $\epsilon \rightarrow 0$ in $L^2(M)$. As in Section 4, additional arguments are required for the lower order terms. Nevertheless, for the cases $k+k' = 1$ we need to argue more carefully since we now require these terms to be of order $o_{L^2,\epsilon}(1)$. For this we will make use of the construction of $\{ b_{j'} \}_{1 \leq j' \leq N}$ introduced in Lemma \ref{special amplitude}.
\begin{lemma}
We can choose sequences $ \{ Q^{+}_{j', \epsilon} \}_{1 \leq j' \leq N} $ and $ \{ Q^{-}_{j', \epsilon} \}_{1 \leq j' \leq N} $ so that
\begin{alignat*}{2}
\lim_{\lambda \rightarrow \infty} \frac{2 \lambda \mathbf{1}_{\Omega_0}(p_0)}{\pi} I_{1,0} = \lim_{\lambda \rightarrow \infty}  \frac{2 \lambda \mathbf{1}_{\Omega_0}(p_0)}{\pi} I_{0,1} = o_{L^2, \epsilon}(1)
\end{alignat*}
where $o_{L^2, \epsilon}(1) \rightarrow 0$ as $\epsilon \rightarrow 0$ in $L^2(M)$.
\end{lemma}
\begin{proof}
Extending $V_1, V_2$ to $M$, since $\tilde{\rho}$ is compactly supported in $M_0'$, we can write $I_{1,0}$ as
\begin{alignat}{2}
\begin{split} \label{first split lower order identification}
I_{1,0} = \sum_{1 \leq j' \leq N} \int_{M} V & \bar{a} T \mathbf{1}_{\tilde{\Omega}_{j'}} e^{2i \psi \lambda} \tilde{\rho} \Big( T^{\star}V_1 a - \sum_{0 \leq k \leq N} Q^{+}_{k, \epsilon}(p_0) b_{k} \Big) \hspace*{0.5mm} d\text{v}_{g}   \\ & 
+  \sum_{j \geq 0} \int_{M} V \bar{a} T \mathbf{1}_{\tilde{M}} \chi_{j} e^{2i \lambda \psi} \tilde{\rho} \Big( T^{\star}V_1 a - \sum_{0 \leq k \leq N} Q^{+}_{k, \epsilon}(p_0) b_{k} \Big) \hspace*{0.5mm} d\text{v}_g.
\end{split}
\end{alignat}
Since $\tilde{\chi}_j$ is identically $1$ on $\tilde{\Omega}_{j'}$, for each $0 \leq j' \leq N$ we have that 
\begin{alignat}{2} \label{two differences}
\begin{split}
 & \hspace*{-6.7cm} \int_{M} V \bar{a} T \mathbf{1}_{\tilde{\Omega}_{j'}} e^{2i \psi \lambda} \tilde{\rho} \Big( T^{\star}V_1 a - \sum_{0 \leq k \leq N} Q^{+}_{k,\epsilon}(p_0) b_{k} \Big) \hspace*{0.5mm} d\text{v}_{g}  \\
=   \int_{M} V \bar{a}  \tilde{\chi}_{j'}' R( \mathbf{1}_{\tilde{\Omega}_{j'}}  e^{2i \psi \lambda} \tilde{\chi}_{j'} \tilde{\rho} T^{\star} V_1 a )  \hspace*{0.5mm} d\text{v}_{g} & - \sum_{0 \leq k \leq N} Q^{+}_{k,\epsilon}(p_0) \int_{M} V \bar{a} \tilde{\chi}_{j'}' R ( \mathbf{1}_{\tilde{\Omega}_{j'}} e^{2i \psi \lambda} \tilde{\chi}_{j'} \tilde{\rho} b_{k}) \hspace*{0.5mm} d\text{v}_g  \\
 +  \int_{M} V\bar{a} K_{j'}( \mathbf{1}_{\tilde{\Omega}_{j'}} e^{2i \psi \lambda} \tilde{\chi}_{j'} \tilde{\rho} T^{\star}V_1 a) \hspace*{0.5mm} d\text{v}_g & - \sum_{0 \leq k \leq N} Q^{+}_{k,\epsilon}(p_0) \int_{M} V \bar{a}  K_{j'}( \mathbf{1}_{\tilde{\Omega}_{j'}}  e^{2i \psi \lambda} \tilde{\chi}_{j'} \tilde{\rho} b_{k}) \hspace*{0.5mm} d\text{v}_{g}.
\end{split}
\end{alignat}
We analyse the two differences in (\ref{two differences}). By making the change of variables $\Phi_{|_{\tilde{\Omega}_{j'}}}(p) = z_{j'}$ and Fubini's theorem, for each $j'$ we have
\begin{alignat*}{2}
\mathbf{1}_{\Omega_0}(p_0) & \int_{M} V \bar{a} \tilde{\chi}_j' R( \mathbf{1}_{\tilde{\Omega}_{j'}} e^{2i \psi \lambda} \tilde{\chi}_{j'} \tilde{\rho} T^{\star} V_1 a ) d\text{v}_g \\
& = - \mathbf{1}_{D_r}(z_0) \int_{D_r} R(V \bar{a} \tilde{\chi}_{j'}' |g_{j'}|^{1/2}) e^{2i \text{Re} \hspace*{0.5mm} (z_{j'} - z_0)^2 \lambda } \tilde{\chi}_{j'} \tilde{\rho} T^{\star} (V_1 a) \hspace*{0.5mm} dz_{j'} d\bar{z}_{j'} \\
& = - \sum_{\mu,\mu' \geq 0} \mathbf{1}_{D_r}(z_0) \bar{z}_{0}^{\mu} z_0^{\mu'} \int_{D_r} R(V \bar{a}_{\mu} \tilde{\chi}_{j'}' |g_{j'}|^{1/2}  )  e^{2i \text{Re} \hspace*{0.5mm} (z_{j'} - z_0)^2 \lambda}\tilde{\chi}_{j'} \tilde{\rho} T^{\star}(V_1 a_{\mu'}) \hspace*{0.5mm} dz_{j'} d\bar{z}_{j'},
\end{alignat*}
where the final expansion in the above can be justified with Lemma \ref{Taylor expansion Lemma} and Sobolev embeddings. We also made the identification $\tilde{\chi}_{j'} T^{\star}(V_1a) = \tilde{\chi}_{j'} T^{\star}(V_1 a) dz_{j'} $ on local charts. Recall from (\ref{this is L^2}) that
\begin{alignat*}{2}
\| R(V \bar{a}_{\mu} \tilde{\chi}_{j'}' |g_{j'}|^{1/2} ) \tilde{\chi}_{j'} \tilde{\rho} T^{\star}(V_1 a_{\mu'}) \|_{L^2(D_r)} \lesssim \| a_{\mu} \|_{L^{\infty}(M_0)} \| a_{\mu'} \|_{L^{\infty}(M_0)} \| V \|_{L^{4/3}(M_0)} \| V_1 \|_{L^{4/3}(M_0)}
\end{alignat*}
can be bounded independently of $\mu, \mu' \geq 0$. Thus we can apply Lemma \ref{stationary phase asym} to find functions of order $o_{L^2,\lambda}^{(\mu,\mu')}(1)$ so that
\begin{alignat}{2} \label{identification first convergence estimate}
\begin{split}
\Big\| & \frac{2 \lambda}{\pi}  \int_{M} V \bar{a} \tilde{\chi}_{j'}' R( \mathbf{1}_{D_r} e^{2i \text{Re} \hspace*{0.5mm} (z_{j'} - z_0)^2 \lambda} \tilde{\chi}_{j'} \tilde{\rho} T^{\star} V_1 a) \hspace*{0.5mm} d\text{v}_{g} \\
& \hspace*{2cm} + \sum_{\mu,\mu' \geq 0} \bar{z}_0^{\mu} z_{0}^{\mu'} {[R(V \bar{a}_{\mu} \tilde{\chi}_{j'}' |g_{j'}|^{1/2} ) \tilde{\chi}_{j'} \tilde{\rho} T^{\star}(V_1 a_{\mu'})]}_{|_{z_{j'} = z_0}} \hspace*{0.5mm} \Big\|_{L^2(D_r)} \\[2mm]
& \lesssim \sum_{\mu,\mu' \geq 0} r^{\mu+\mu'} \| o_{L^2, \lambda}^{(\mu,\mu')}(1) \|_{L^{2}(D_r)} \\
& \hspace*{1.8cm} \lesssim \Big( \sum_{\mu, \mu' \geq 0} r^{\mu+\mu'} \| a \|_{L^{\infty}(M_0)} \| a_{\mu'} \|_{L^{\infty}(M_0)} \Big)  \| V \|_{L^{4/3}(M_0)} \| V_1 \|_{L^{4/3}(M_0)} < \infty
\end{split}
\end{alignat}
In particular, (\ref{identification first convergence estimate}) implies that
\begin{alignat}{2} 
\begin{split} \label{identification first convergence}
\lim_{\lambda \rightarrow \infty} \frac{2 \lambda \mathbf{1}_{\Omega_0}(p_0)}{\pi} & \sum_{1 \leq j' \leq N}  \int_{M} V \bar{a} \tilde{\chi}_{j'}' R( \mathbf{1}_{\tilde{\Omega}_{j'}} e^{2i \psi \lambda } \tilde{\chi}_{j'} \tilde{\rho} T^{\star}V_1 a ) \hspace*{0.5mm} d\text{v}_{g} \\
& = - \sum_{1 \leq j' \leq N} \sum_{\mu, \mu' \geq 0} \mathbf{1}_{D_r}(z_0)  \bar{z}_0^{\mu} z_0^{\mu'} [R( V \bar{a}_{\mu} \tilde{\chi}_{j'}' |g_{j'}|^{1/2} ) \tilde{\chi}_{j'} \tilde{\rho} T^{\star}(V_1 a_{\mu'})]_{|_{z_{j‘} = z_0}}
\end{split}
\end{alignat}
in $L^2(M)$. On the other hand, on the support of $\tilde{\chi}_{j'}$ we may exploit the Taylor expansion of $b$ introduced in Lemma \ref{Taylor expansion Lemma} so that for every $p_0 \in \Omega_0$, we have
\begin{alignat*}{2}
\tilde{\chi}_{j'}(p) b_{k}(p;p_0) = \sum_{\mu \geq 0} \tilde{\chi}_{j'}(z_{j'}) b_{k,\mu}^{(j')}(z_{j'}) z_0^{\mu} dz_{j'} \ \ \text{for all} \ \  1 \leq j',k \leq N.
\end{alignat*}
By (\ref{estimate for Taylor expansion}) such an expansion satisfies the same convergence property of $a$, thus we have
\begin{alignat*}{2} 
\mathbf{1}_{\Omega_0}(p_0) & Q^{+}_{k,\epsilon}(p_0) \int_{M} V \bar{a}  \tilde{\chi}_{j'}' R( \mathbf{1}_{\tilde{\Omega}_{j'}} e^{2i \psi \lambda} \tilde{\chi}_{j'} \tilde{\rho} b_{k}) \hspace*{0.5mm} d\text{v}_g \\
& = - \sum_{\mu, \mu' \geq 0} \mathbf{1}_{D_r}(z_0) Q^{+}_{k}(z_0) \bar{z}_{0}^{\mu} z_{0}^{\mu'} \int_{D_r} R(V \bar{a}_{\mu} \tilde{\chi}_{j'}' |g_{j'}|^{1/2})  e^{2i (z_{j'}-z_0)^2 \lambda} \tilde{\chi}_{j'} \tilde{\rho} b_{k,\mu'}^{(j')} \hspace*{0.5mm} dz_{j'} d\bar{z}_{j'} 
\end{alignat*}
for each $1 \leq j',k \leq N$. Now arguing in exactly the same way as in (\ref{identification first convergence}), it is easy to see that
\begin{alignat}{2} \label{identification second convergence} 
\begin{split}
& \lim_{\lambda \rightarrow \infty}  \frac{2\lambda \mathbf{1}_{\Omega_0}(p_0) Q^{+}_{k,\epsilon}(p_0)}{\pi} \int_{M} V \bar{a} \tilde{\chi}_{j'}' R( \mathbf{1}_{\tilde{\Omega}_{j'}} e^{2i \psi \lambda} \tilde{\chi}_{j'} \tilde{\rho} b_{k}) \hspace*{0.5mm} \hspace*{0.5mm} d\text{v}_g \\
&  =  - \sum_{\mu \geq 0} \sum_{\mu' \geq 0} \mathbf{1}_{D_r}(z_0) Q^{+}_{k,\epsilon}(z_0)  \bar{z}_{0}^{\mu} z_{0}^{\mu'} [R(V \bar{a}_{\mu} \tilde{\chi}_{j’}' |g_{j‘}|^{1/2}) ]_{|_{z_{j’} = z_0}} ( \tilde{\chi}_{j'} \tilde{\rho} )(\Phi_{|_{\tilde{\Omega}_{j'}}}^{-1}(z_0)) b_{k,\mu'}^{(j')}( \Phi_{|_{\tilde{\Omega}_{j'}}}^{-1}(z_0) ) \\[1mm]
& = - \sum_{\mu \geq 0}  \mathbf{1}_{D_r}(z_0) Q^{+}_{k,\epsilon}(z_0)  \bar{z}_0^{\mu} [R(V \bar{a}_{\mu} \tilde{\chi}_{j'}' |g_{j'}|^{1/2}  ) ]_{|_{z_{j'} = z_0}} ( \tilde{\chi}_{j'} \tilde{\rho})(\Phi_{|_{\tilde{\Omega}_{j'}}}^{-1}(z_0)) b_{k}^{(j')}( \Phi_{|_{\tilde{\Omega}_{j'}}}^{-1}(z_0) ; \Phi_{|_{\Omega_0}}^{-1}(z_0)  )
\end{split}
\end{alignat}
in $L^2(M)$. By formula (\ref{locla condition for b}) we obviously have
\begin{alignat}{2} \label{delta condition for b}
 b^{(j')}_{k}( \Phi_{|_{\tilde{\Omega}_{j'}}}^{-1}(z_0) ; \Phi_{|_{\Omega_0}}^{-1}(z_0) )  =
\begin{cases}
1 \ \ \text{if} \ \ j' = k, \ \ \text{and} \\
0 \ \ \text{if} \ \ j' \neq k
\end{cases}
\end{alignat}
for all $z_0 \in D_r$ and $1 \leq j',k \leq N$. Summing over all such $j',k$ we see from (\ref{delta condition for b}) that
\begin{alignat}{2} 
\begin{split} \label{second convergence}
\lim_{\lambda \rightarrow \infty} \sum_{1 \leq j',k \leq N} & \frac{2 \lambda \mathbf{1}_{\Omega_0}(p_0) Q^{+}_{k,\epsilon}(p_0)}{\pi} \int_{M} V \bar{a} \tilde{\chi}_{j'}' R( \mathbf{1}_{\tilde{\Omega}_{j'}}e^{2i \psi \lambda} \tilde{\chi}_{j'}  b_{k} ) \hspace*{0.5mm} d\text{v}_g \\
& = -  \sum_{1 \leq j' \leq N} \sum_{\mu \geq 0} \mathbf{1}_{D_r}(z_0) Q^{+}_{j',\epsilon}(z_0) ( \tilde{\chi}_{j'} \tilde{\rho} )(\Phi_{|_{\tilde{\Omega}_{j'}}}^{-1}(z_0))  \bar{z}_0^{\mu} [R( V \bar{a}_{\mu}  \tilde{\chi}_{j'}' |g_{j'}|^{1/2} )]_{|_{z_{j'} = z_0}}.
\end{split}
\end{alignat}
To take care of the smoothing terms in (\ref{two differences}), we apply the exact same procedure to get that
\begin{alignat}{2} \label{third convergence}
\begin{split}
& \lim_{\lambda \rightarrow \infty} \frac{2 \lambda \mathbf{1}_{\Omega_0}(p_0) }{\pi} \sum_{1 \leq j' \leq N}  \int_{M} V \bar{a} K_{j'} ( \mathbf{1}_{\tilde{\Omega}_{j'}}  e^{2i \psi \lambda} \tilde{\chi}_{j'} \tilde{\rho}  T^{\star} V_1 a) \hspace*{0.5mm} d\text{v}_{g} \\
& \hspace*{3cm} = \sum_{1 \leq j' \leq N} \sum_{\mu,\mu' \geq 0} \mathbf{1}_{D_r}(z_0) \bar{z}_{0}^{\mu} z_{0}^{\mu'} [K_{j'}^{'}(V \bar{a}_{\mu} ) \tilde{\chi}_{j'} \tilde{\rho} T^{\star}(V_1 a_{\mu'}) |g_{j'}|^{1/2}]_{|_{z_{j'} = z_0}},  \\
& \lim_{\lambda \rightarrow \infty} \frac{2 \lambda \mathbf{1}_{\Omega_0}(p_0)}{\pi} \sum_{1 \leq j', k \leq N} Q^{+}_{k,\epsilon}(p_0) \int_{M} V \bar{a} K_{j'}( \mathbf{1}_{\tilde{\Omega}_{j'}} \tilde{\chi}_{j'} e^{2i \psi \lambda} \tilde{\rho} b_{k} ) \hspace*{0.5mm} d\text{v}_{g} \\
& \hspace*{3cm} = \sum_{1 \leq j' \leq N} \sum_{\mu \geq 0} \mathbf{1}_{D_r}(z_0)  Q^{+}_{j',\epsilon}(z_0) \bar{z}_{0}^{\mu} [  K_{j'}'(V \bar{a}_{\mu}) \tilde{\chi}_{j'} \tilde{\rho} |g_{j'}|^{1/2}]_{|_{z_{j'}=z_0}}
\end{split}
\end{alignat}
in $L^{2}(M)$ where $K_{j'}'$ has smooth kernel. By Sobolov's inequality, we have
\begin{alignat*}{2}
\sum_{\mu' \geq 0} \| z_{0}^{\mu'} [T^{\star}(V_1a_{\mu'})]_{|_{z_{j'} = z_0}} \|_{L^4} \lesssim \sum_{\mu' \geq 0} r^{\mu'} \| T^{\star}(V_1 a_{\mu'}) \|_{W^{1. 4/3}} \lesssim \sum_{\mu' \geq 0} \| V_1 \|_{L^{4/3}} r^{\mu'} \| a_{\mu'} \|_{L^{\infty}}  < \infty.
\end{alignat*}
Since $\{ Q^{+}_{j',\epsilon} \}_{1 \leq j' \leq N} \subset \mathcal{C}^{\infty}_{c}(M)$ were fixed arbitrarily, putting (\ref{identification first convergence}), (\ref{second convergence}) and (\ref{third convergence}) together, we see from (\ref{two differences}) that if we choose them to be smooth approximations such that
\begin{alignat*}{2}
\lim_{\epsilon \rightarrow 0} \mathbf{1}_{D_r}(z_0) Q^{+}_{j', \epsilon}(z_0) = \sum_{ \mu' \geq 0 } \mathbf{1}_{D_r}(z_0) z_0^{\mu'} [  T^{\star}(V_1 a_{\mu'}) ]_{z_{j'} = z_0}, \ \ 1 \leq j' \leq N
\end{alignat*}
in $L^2(M)$, then we have 
\begin{alignat*}{2}
\lim_{\lambda \rightarrow \infty}  \sum_{1 \leq j' \leq N} \int_{M} V \bar{a} T \mathbf{1}_{\Omega_{j'}'} e^{2i \psi \lambda} \tilde{\rho} T^{\star}(V_1 a) \hspace*{0.5mm} d\text{v}_{g} = o_{L^2, \epsilon}(1)
\end{alignat*}
where $\lim_{\epsilon \rightarrow 0} o_{L^2, \epsilon}(1) = 0$ in $L^{2}(M)$. By construction, $\tilde{M}$ contains no point $p \in M_0$ for which we have $\Phi(p) \in D_r$. By taking $M_0'$ small enough, we may assume without loss of generality that $\text{Supp} \hspace*{0.5mm} \tilde{\rho}$ is disjoint from those neighbourhood $\Omega \subset M \backslash M_0$ such that $\Phi(\Omega) \cap D_r \neq \emptyset$. Thus $ \Phi( \tilde{M} \hspace*{0.5mm} \cap \hspace*{0.5mm} \text{Supp} \hspace*{0.5mm} \tilde{\rho}) $ and $D_r$ are disjoint, and the same procedure yields
\begin{alignat}{2} \label{final vanishing integral}
\lim_{\lambda \rightarrow \infty} \int_{M} V \bar{a} T \mathbf{1}_{\tilde{M}} \chi_{j} e^{2i \psi \lambda} \tilde{\rho} \Big( T^{\star}V_1 a  - \sum_{1 \leq k \leq N} Q^{+}_{k,\epsilon}(p_0) b_{k} \Big) \hspace*{0.5mm} d\text{v}_{g} = 0, \ \ j \geq 0.
\end{alignat}
Indeed, it is sufficient to split $T\chi_{j}$ into linear combinations of $\chi_{j}'R \chi_{j'}$ and $K_{j}$ and apply Lemma \ref{weighted main theorem expansion} and Fubini's theorem as before, at which point the expression $\mathbf{1}_{D_r} \mathbf{1}_{ \Phi ( \tilde{M} \hspace*{0.5mm} \cap \hspace*{0.5mm} \text{Supp} \hspace*{0.5mm} \tilde{\rho} \hspace*{0.5mm} \cap \hspace*{0.5mm} \Omega_j) }$ appears in the resulting limits, and we conclude from the remarks above that (\ref{final vanishing integral}) vanishes. The claim for the case of $I_{1,0}$ now follows from expression (\ref{first split lower order identification}). The obvious modifications holds for the case of $I_{0,1}$.
\end{proof}

Finally we show that
\begin{alignat}{2} \label{integral identity 100}
\lim_{\lambda \rightarrow \infty} \frac{2 \lambda \mathbf{1}_{\Omega_0}(p_0)}{\pi} \sum_{k+k' \geq 2} I_{k,k'} = 0
\end{alignat}
in $L^2(M)$. For this we rely only on the Carleman estimates derived in Section 3. 
\begin{lemma} \label{identification lemma 1}
There exists constants $C, \tilde{C} > 0$ independent of $\lambda$ and $p_0$ such that 
\begin{alignat}{2} \label{identification lemma equation 1}
\begin{split}
| & I_{k,k'} |  \leq \frac{\tilde{C}}{\lambda^{1+}} \Big( \frac{C \max\{ \| V_1 \|_{L^{p}(M_0)}, \| V_2 \|_{L^{p}(M_0)} \} }{\lambda^{0+}} \Big)^{k+k' - 4}
\end{split}
\end{alignat}
for all $k+k' \geq 4$. If $2 \leq k+k' < 4$, then we have
\begin{alignat}{2} \label{identification lemma equation 2}
|I_{k,k'}| \leq \frac{\tilde{C}}{\lambda^{1+}}.
\end{alignat}
\end{lemma}
\begin{proof}
Let $p'$ be defined by $ 1/p' = 1/p - 1/2 $. Set $1/q = 1/p + 1/p' = 2/p -1/2$. Since $ p \in {]}4/3, 2{[} $ we can choose $r \in {]}2, 4{[}$ by $1/r = 1/p - 1/4$ so that $1/2 + 1/r \geq 1/q > 1/2$. Hence if $k' = 0$, then by Corollary \ref{Carleman estimate corollary} we have
\begin{alignat*}{2}
|I_{k,0}| \leq \int_{M_0} |V||u_{k}| & |a| \hspace*{0.5mm} d\text{v}_{g} \lesssim \| a V \|_{L^{2}} \| T_{\Psi} \tilde{\rho} T^{\star}( V_1 u_{k-1} ) \|_{L^{r}} \\
& \lesssim \frac{ \| a \|_{L^{\infty}} \| V \|_{L^2} \| V_1 u_{k-1} \|_{L^{q}} }{\lambda^{1 - ( \frac{1}{q} - \frac{1}{r} ) }} \leq  \frac{ \| a \|_{L^{\infty}} \| V \|_{L^2} \| V_1 \|_{L^{p}} \| u_{k-1} \|_{L^{p'}}  }{\lambda^{ 1 - ( \frac{1}{q} - \frac{1}{r} )  }}.
\end{alignat*}
If $k=2$, then we can directly estimate from Sobolev embedding and (\ref{dbar estimate}) that
\begin{alignat}{2} \label{middle terms identification estimate 1}
\begin{split}
\| u_{k-1} \|_{L^{p'}} & \lesssim \| T_{\Psi} \tilde{\rho} T^{\star}( V_1 a ) \|_{L^{p'}} + \sum_{1 \leq j' \leq N} \|Q^{+}_{j',\epsilon}\|_{L^{\infty}} \| T_{\Psi} \tilde{\rho} b_{j'} \|_{L^{p'}} \\
& \hspace*{1cm} \lesssim  \frac{ \| a \|_{L^{\infty}} \| V_1 \|_{L^{p}} }{\lambda^{1 - (\frac{1}{p} - \frac{1}{p'}) }} + \frac{ \max_{1 \leq j' \leq N} \{ \| Q^{+}_{j',\epsilon} \|_{L^{\infty}} \| b_{j'} \|_{L^{\infty}} \} }{ \lambda^{ 1 - (\frac{1}{p} - \frac{1}{p'}) } }. 
\end{split}
\end{alignat}
Alternatively if $k \geq 3$, then we iterate the $L^{p} \rightarrow L^{\infty}$ estimate obtained in (\ref{Carleman estimate for the Laplacian}) to get that there exists $C >0$ such that
\begin{alignat}{2} \label{middle terms identification estimate 2}
\begin{split}
& \|  u_{k-1} \|_{L^{p'}} \leq \frac{\| V_1 \|_{L^{p}} \| u_{k-2} \|_{L^{\infty}} }{\lambda^{1- ( \frac{1}{p} - \frac{1}{p'} ) }} \\
& \lesssim   \frac{\| a \|_{L^{\infty}} \| V_1 \|_{L^{p}}}{\lambda^{1 - ( \frac{1}{p} - \frac{1}{p'} )}}  \Big( \frac{C \| V_1 \|_{L^{p}} }{\lambda^{0+}} \Big)^{k-2} + \frac{ \max_{1 \leq j' \leq N} \{ \| Q^{+}_{j', \epsilon} \|_{L^{\infty}} \| b_{j'} \|_{L^{\infty}} \}}{\lambda^{ 1 - ( \frac{1}{p} - \frac{1}{p'} ) }} \Big( \frac{ C \| V_1 \|_{L^{p}}}{\lambda^{0+}} \Big)^{k-3}.
\end{split}
\end{alignat}
A similar calculation works for the case $k=0$. \par 
In the cases where $k, k' \neq 0$, we can apply H\"older's inequality to see that 
\begin{alignat*}{2} \begin{split}
\int_{M_0} |V| |u_{k} | |v_{k'}| \hspace*{0.5mm} d\text{v}_{g} & \lesssim \| V \|_{L^2} \| T_{\Psi} \tilde{\rho} T^{\star}(V_1 u_{k-1}) \|_{L^4} \| \bar{T}_{\Psi} \bar{T}^{\star}( V_2 v_{k'-1} ) \|_{L^{4}} \\
& \lesssim \frac{\| V \|_{L^{2}} \| V_1 \|_{L^{p}} \| V_2 \|_{L^{p}}}{\lambda^{1+}} \| u_{k-1} \|_{L^{\infty}} \| v_{k'-1} \|_{L^{\infty}}.
\end{split}
\end{alignat*}
As above we can now enumerate the $L^{p} \rightarrow L^{\infty}$ estimates to get that
\begin{alignat}{2} \label{middle terms identification estimate 3}
\begin{split} 
 \| u_{k-1} \|_{L^{\infty}}   \lesssim \| a \|_{L^{\infty}} \| V_1 \|_{L^{p}}  \Big( \frac{C \| V_1 \|_{L^{p}} }{\lambda^{0+}} \Big)^{k-1} & +  \max_{1 \leq j' \leq N} \{ \| Q^{+}_{j', \epsilon} \|_{L^{\infty}} \| b_{j'} \|_{L^{\infty}} \} \Big( \frac{ C \| V_1 \|_{L^{p}}}{\lambda^{0+}} \Big)^{k-2},  \\
 \| v_{k'-1} \|_{L^{\infty}}  \lesssim  \| a \|_{L^{\infty}} \| V_2 \|_{L^{p}} \Big( \frac{C \| V_2 \|_{L^{p}} }{\lambda^{0+}} \Big)^{k'-1} &+ \max_{1 \leq j' \leq N} \{ \| Q^{-}_{j', \epsilon} \|_{L^{\infty}} \| b_{j'} \|_{L^{\infty}} \} \Big( \frac{ C \| V_2 \|_{L^{p}}}{\lambda^{0+}} \Big)^{k'-2}
\end{split}
\end{alignat}
for all $k,k' \geq 2 $. In particular, putting together estimates (\ref{middle terms identification estimate 2}) and (\ref{middle terms identification estimate 3}) implies (\ref{identification lemma equation 1}) for sufficiently large $\lambda$. On the other hand, if $k = k' = 1$, then
\begin{alignat}{2} \label{middle terms identification estimate 4}
\begin{split}
\int_{M_0} |V| |u_{1}| |v_{1}| \hspace*{0.5mm} d\text{v}_{g} \leq \frac{\| V \|_{L^2}}{\lambda^{1+}} & \Big( \| V_1 \|_{L^p} \| a \|_{L^{\infty}} + \max_{1 \leq j' \leq N} \| Q^{+}_{j', \epsilon} \|_{L^{\infty}} \| b_{j'} \|_{L^{\infty}} \Big) \\
& \hspace*{1cm} \times \Big( \| V_2 \|_{L^{p}}\| a \|_{L^{\infty}} + \max_{1 \leq j' leq N} \| Q^{-}_{j', \epsilon} \|_{L^{\infty}} \| b_{j'} \|_{L^{\infty}}\Big).
\end{split}
\end{alignat}
Combining (\ref{middle terms identification estimate 1}), (\ref{middle terms identification estimate 4}) and the other estimates now implies (\ref{identification lemma equation 2}).
\end{proof}

By applying the estimates in Lemma \ref{identification lemma 1}, we easily see that
\begin{alignat}{2}  \label{final estimate}
\begin{split} 
\Big| \frac{2 \lambda \mathbf{1}_{\Omega_0}(p_0)}{\pi}  \sum_{k+k' \geq 2} I_{k,k'} \Big|  \lesssim \frac{\tilde{C}}{\lambda^{0+}} + \frac{\tilde{C}}{\lambda^{0+}} \sum_{k+k' \geq 4} \Big( \frac{C \max\{ \| V_1 \|_{L^{p}(M_0)}, \| V_2 \|_{L^{p}(M_0)} \} }{\lambda^{0+}} \Big)^{k+k'-4}.
\end{split}
\end{alignat}
The last term in (\ref{final estimate}) converges for sufficiently large $\lambda > \lambda_0$ and is in particular of order $o(1)$ as $\lambda \rightarrow \infty$. Since $C$ and $\tilde{C}$ are independent of $p_0$, we can take this limit in (\ref{weighted main theorem expansion}) and deduce from the sequence of lemmas proved above that
\begin{alignat*}{2}
\mathbf{1}_{\Omega_0} V |g|^{1/2} = o_{L^2, \epsilon}(1)
\end{alignat*}
for all $\epsilon > 0$. Letting $\epsilon \rightarrow 0$ now yields $\mathbf{1}_{\Omega_0} V |g|^{1/2} = 0$. Since $|g|$ is non-vanishing we must therefore have $V = 0$ almost everywhere on $\Omega_0$. This concludes the proof of the claim.
\end{proof}

\medskip

\noindent {\bf Acknowledgments.} This work is completed as part of the author's HDR degree while under the support of A/Prof Tzou's projects ARC DP190103451 and ARC DP190103302. The author is grateful for the supervision he has received.


\medskip

 \begin{thebibliography}{99}
\bibitem {bukcalderon}
A. L. Bukhgeim, 
{Recovering a potential from Cauchy data in the two-dimensional case},
{\it Journal of Inverse and Ill-posed Problems}, {\bf 16}(1), 2008, pages 19-33.

\bibitem {L2calderon}
E. Bl\aa sten, O.Y. Imanuvilov and M. Yamamoto,
{Stability and uniqueness for a two-dimensional inverse boundary value problem for less regular potentials}, 
{\it Inverse Problems and Imaging}, {\bf 9}(3), 2015, pages 709-723.

\bibitem {L4/3calderon}
E. Bl\aa sten, L. Tzou and J.-N. Wang, 
{Uniqueness for the inverse boundary value problem with singular potentials in 2D}, 
{\it Mathematische Zeitschrift}, 2019.

\bibitem {Ln/2Euclidean}
S. Chanillo,
{A Problem in Electrical Prospection and a n-Dimensional Borg-Levinson Theorem},
{\it Proceedings of the American Mathematical Society}, {\bf 108}(3), 1990, pages 761-767.

\bibitem {Ln/2partialdataeuclidean}
F.J. Chung and L. Tzou,
{The $L^{p}$ Carleman estimate and a partial data inverse problem},
{\it arXiv Preprint}, 2016, arXiv:1610.01715.

\bibitem {Ln/2admissible}
D. Dos Santos Ferreira, C. E. Kenig and M. Salo,
{Determining an unbounded potential from Cauchy data in admissible geometries}, 
{\it Communications in Partial Differential Equations}, {\bf 38}(1), 2013, pages 50-68.

\bibitem {AMCalderon}
D. Dos Santos Ferreira, C. E. Kenig, M. Salo and G. Uhlmann,
{Limiting Carleman Weights and Anisotropic Inverse Problems}, 
{\it Inventiones mathematicae}, {\bf 178}(1), 2009, pages 119-171.

\bibitem {holofunction}
R.C. Gunning and R.Narasimhan,
{Immersion of open Riemann surfaces}, 
{\it Mathematische Annalen}, {\bf 174}(2), 1967, pages 103-108.

\bibitem {2dpartialdataold}
O.Y. Imanuvilov, G. Uhlmann and M. Yamamoto, 
{The Calder\'on problem with partial data in two dimensions}, 
{\it Journal of the American Mathematical Society}, {\bf 23}(3), 2010, pages 655-691.

\bibitem {2dpartialdatanew}
O.Y. Imanuvilov and M. Yamamoto,
{Inverse boundary value problem for Schrodinger equation in two dimensions},
{\it SIAM Journal on Mathematical Analysis}, {\bf 44}(3), 2012, pages 1333-1339.

\bibitem {partialdataeuclidean}
C.E. Kenig, J. Sj\"oestrand and G. Uhlmann,
{The Calder\'on problem with partial data}, 
{\it Annals of Mathematics}, {\bf 165}(2), 2007, pages 567-591.

\bibitem {nachmancalderon} 
A. Nachman, 
{Reconstructions from boundary measurements}, 
{\it Annals of Mathematics}, {\bf 128}(3), 1988, pages 531-576.

\bibitem {improveregularnew}
V.S. Serov and L. P\"aiv\"arinta, 
{New estimates of the Green-Faddeev function and recovering of singularities in the two-dimensional Schr\"odinger operator with fixed energy}, {\bf 21}(4), 2005, pages 1291-1301.

\bibitem {calderonfirstresult}
J. Sylvester and G. Uhlmann, 
{A global uniqueness theorem for an inverse boundary value problem}, 
{\it Annals of Mathematics}, {\bf 125}(1), 1987, pages 153–169.	

\bibitem {improveregularold}
Z. Sun and G. Uhlmann,
{Recovery of singularities for formally determined inverse problems},
{\it Communications in Mathematical Physics},
{\bf 153}(3), 1993, pages 431-445.

\bibitem {Ln/2partialdatamanifold}
L. Tzou,
{Partial data Calder\'on problems on $L^{n/2}$ potentials n admissible manifolds}, 
{\it arXiv Preprint}, 2018, arXiv:1805.09161.

\bibitem {reflectionpaper}
L. Tzou,
{The reflection principle and Calder\'on problems with partial data},
{\it Mathematische Annalen}, {\bf 369}(1-2), 2017, pages 913-956.

\bibitem {leocalderon}
L. Tzou and C. Guillarmou, 
{Calder\'on inverse problem with partial data on Riemann surfaces}, 
{\it Duke Mathematical Journal}, {\bf 158}(1), 2011, pages 83-120.

\bibitem {leoconnection}
L. Tzou and C. Guillarmou,
{Identification of a connection from Cauchy data on a Riemann surface with boundary},
{\it Geometric and Functional Analysis}, {\bf 21}(2), 2011, pages 393-418.


  \end{thebibliography}
 \end{document}